\documentclass[11pt]{amsart}
\usepackage{amsfonts,amssymb,amsmath,euscript}
\usepackage{enumitem,hyperref}

\setlist[enumerate]{label={\upshape(\roman*)}}
\usepackage{verbatim}
  
\theoremstyle{plain}
 \newtheorem{thm}{Theorem}[section]
 \newtheorem{prop}[thm]{Proposition}
 \newtheorem{lemma}[thm]{Lemma}
 \newtheorem{cor}[thm]{Corollary}
\theoremstyle{definition}
 \newtheorem{defn}[thm]{Definition}
 \newtheorem{assum}[thm]{Assumption}
 \newtheorem{remark}[thm]{Remark}

\newcommand{\mbR}{\mathbb{R}}
\newcommand{\mbC}{\mathbb{C}}
\newcommand{\mbD}{\mathbb{D}}
\newcommand{\mbN}{\mathbb{N}}
\newcommand{\mbT}{\mathbb{T}}

\newcommand{\bfA}{\mathbf{A}}
\newcommand{\bfB}{\mathbf{B}}

\newcommand{\euR}{\EuScript{R}}

\newcommand{\clA}{\mathcal{A}}

\newcommand{\clH}{\mathcal{H}}
\newcommand{\clK}{\mathcal{K}}

\newcommand{\clV}{\mathcal{V}}

\newcommand{\Tr}{\operatorname{Tr}}
\renewcommand{\Im}{\operatorname{Im}}

\newcommand{\dom}{\operatorname{dom}} 
\newcommand{\im}{\operatorname{im}}
\newcommand{\rank}{\operatorname{rank}}
\newcommand{\supp}{\operatorname{supp}}
\newcommand{\sgn}{\operatorname{sgn}}
\newcommand{\scal}[1]{\left\langle #1 \right\rangle}
\newcommand{\eps}{\varepsilon}
\renewcommand{\phi}{\varphi}
\newcommand{\rind}{\operatorname{ind}_{res}}
\newcommand{\coker}{\operatorname{coker}}

\renewcommand{\iff}{\Leftrightarrow} 

\newcommand{\lra}{\leftrightarrow}
\newcommand{\Rset}{\euR(z_0)}

\makeatletter
\@namedef{subjclassname@2020}{\textup{2020} Mathematics Subject Classification}
\makeatother

\begin{document}
\title[Coupling resonances and spectral properties of $(H_0-z)^{-1}V$]{Coupling resonances and spectral properties of the product of resolvent and perturbation}
\author{Nurulla Azamov and Tom Daniels}
%\address{College of Science and Engineering
%   \\ Flinders University 
%   \\ South Rd, Tonsley, SA 5042 Australia}
%\email{nurulla.azamov@flinders.edu.au}
 \keywords{branching of coupling resonance points, order of eigenpath, tangency to resonance set}
 \subjclass[2020]{ %Mathematics Subject Classification (2020).
     Primary 47A11; %Local spectral properties of linear operators
     Secondary 
     %47A10, %Spectrum, resolvent
     47A55, %Perturbation theory of linear operators     
     %47A40, %Scattering theory for linear operators
     %47A70, %(Generalized) eigenfunction expansions of linear operators; rigged Hilbert spaces
     %81Q10, %Selfadjoint operator theory in quantum theory, including spectral analysis
     81U24. %Resonances in quantum scattering theory
     %81U99, %Quantum scattering theory; none of the above, but in this section     
 }
\begin{abstract}
Given a self-adjoint operator~$H_0$ and a relatively compact self-adjoint perturbation~$V$, we study in some detail the spectral properties of the product~$(H_0-z)^{-1}V$, $z\in\mbC$.
For~some numbers $r_z,$ the eigenvalues of $(H_0 + s V - z)^{-1}V$, $s\in\mbC$, are $(s - r_z)^{-1}$. 
We study the root spaces of the eigenvalues $(s - r_z)^{-1}$ and complex-analytic properties of the functions~$r_z$ such as branching points. 
In particular, for a generic case, we give a variety of necessary and sufficient conditions for branching. 
%This work is motivated by the ubiquity of the product $(H_0 + s V - z)^{-1}V$ in spectral theory. 
The functions $r_z,$ called coupling resonances, are important in the spectral analysis of $H_0 + r V$ for any real number $r.$ 
For instance, they afford a description of the spectral shift function (SSF) of the pair $H_0$ and $V,$ as well as the absolutely 
continuous and singular parts of the SSF. 
%Nevertheless, it appears that in the existing literature the functions $r_z$ did not receive a thorough study they deserve. This paper aims to fill in this gap. 
A thorough study of real-valued coupling resonances $r_\lambda$ for real~$\lambda$ outside of the essential spectrum was carried out in a recent work by the first author. 
Here we extend this study to the complex domain, motivated by the fact, which is well known in the case of a rank-one perturbation, that the behaviour of coupling resonances~$r_z$ near the essential spectrum provides valuable information about the latter.% and has numerous important applications.
%and points where their derivatives become zero or infinity, 
% and we show that under mild assumptions on $H_0$ and~$V$
% the functions $r_z^j$ do not get absorbed by $\infty,$
% that is, eigenvalues of $(H_0- z)^{-1}V$ do not get absorbed by $0.$
\end{abstract}
\maketitle

%\begin{center} {\tiny \sf \today \ draft v\,13.1} 
%\end{center}

\section{Introduction}
The product 
\begin{equation} \label{F: Rz(H_0)V}
  R_z(H_0)V = (H_0-z)^{-1}V,
\end{equation}
where~$H_0$ is a self-adjoint operator and~$V$ is a~$H_0$-compact self-adjoint operator, is ubiquitous within perturbation and spectral theories, appearing for example in the second resolvent identity, Neumann series and its numerous applications, perturbation determinants, the stationary formula for the scattering matrix, and in scattering theory more generally. 
%Of course, all these appearances are connected to one another. 
%Further, since~$R_z(H_0)V$ is compact, one can approximate it by finite-rank operators, and this circumstance is used for numerical calculations in practical applications of scattering theory. 
A detailed study of the spectral properties of this operator is of interest, yet we feel that the subject has not received the attention it deserves. 
This paper arose more specifically because of the relevance of the spectral properties of~\eqref{F: Rz(H_0)V} to the phenomenon, associated with the pair of self-adjoint operators $H_0$ and $H_1 := H_0 + V$, of the flow of singular spectrum within the aggregate of essential spectrum.
Before discussing this connection and the context it provides, one straightforward way to outline the main result of this paper is as follows.

Since the operator~\eqref{F: Rz(H_0)V} is compact, its spectrum is discrete away from zero, which is the only element of its essential spectrum. 
Further, since it depends holomorphically on~$z,$ the isolated elements of its spectrum are also holomorphic in~$z$ -- everywhere besides potentially any of the discrete set of exceptional points, defined as those points where the number of distinct eigenvalues within a small neighbourhood of a given isolated eigenvalue is different from its otherwise constant value. 
The general theory of e.g. Kato's classic~\cite{Kat80} provides many other properties of these eigenvalues. 
However, we are interested here in those spectral properties that are specific to~\eqref{F: Rz(H_0)V}. 
Three questions which can be posed in this regard are: 
\begin{itemize}[leftmargin=1.5\parindent]
\item Can an eigenvalue get absorbed by, or, which is the same, emerge from, zero? 
\item Are there branching points of the eigenvalues? %and how can they be characterised? 
\item Do the eigenvalues have critical points, that is, points with zero derivative? %, and how can they be described? 
\end{itemize}
The main result of this paper addresses the second question. 
In a limited yet generic case, we establish various equivalent conditions to branching. %(Theorem~\ref{T: criteria for splitting} below).
One characterisation for an eigenvalue $-s^{-1} = -s(z)^{-1}$ of~\eqref{F: Rz(H_0)V} to have a branching point at~$z$ is as follows. 
Let $\phi$ be an eigenvector of 
\[H_s := H_0 + sV \]
corresponding to the eigenvalue~$z$. 
Note that the eigenspaces of~\eqref{F: Rz(H_0)V} and~$H_s$ corresponding to the eigenvalues~$-s^{-1}$ and~$z$ respectively are equal, as is easily seen by multiplying the eigenvalue equations by $H_0 - z$ and $R_z(H_0)$ respectively.
Then $z$ is a branching point of~$s(z)$ iff there exists a `conjugate' vector $\phi^*$, belonging to the eigenspace of $H_{\bar s} = H_0 + \bar sV$ corresponding to~$\bar z$ and satisfying $\scal{\phi^*,\phi} = 1$, such that $\scal{\phi^*, V\phi} = 0.$ 

The limitation of this result is that we assume the existence of a conjugate vector as well as its regular behaviour as a function of $s$; an assumption which holds for example if the eigenvalue~$z$ of $H_s$ is simple, or if~$z$ is semisimple and, when considered as a function of $s$, it does not have an exceptional point at~$s$. 
We believe that this concept of a conjugate vector may prove useful in further study of the essential spectrum. 

\medskip
The eigenvalues of~\eqref{F: Rz(H_0)V} are related to the eigenvalues of 
\begin{equation} \label{F: Rz(H_s)V}
  R_z(H_s)V = (H_0 + sV - z)^{-1}V,
\end{equation}
%where $H_s = H_0 + sV,$ 
by a simple relation: 
there are functions, $r_z^j,$ $j=1,2,\ldots,$ of the variable~$z$ only, 
such that for any~$s\in\mbC$ for which~$z$ belongs to the resolvent set of~$H_s$, the eigenvalues of~\eqref{F: Rz(H_s)V} are 
\[
  (s-r_z^j)^{-1}, \quad j=1,2,\ldots
\]
This relation shows that the functions $r_z^j$ have no less bearing than the eigenvalues themselves. 

If $V$ has rank one, there is only one such function, which appears in different guises, for instance in relation to Weyl $m$-functions in the theory of the Sturm-Liouville equation (see e.g.~\cite{Sim05}) and in the theory of random Schr\"odinger operators (see e.g.~\cite[Theorem~5.3]{AizWar}). 
If~$\rank(V) = 1$ and also $V\geq 0,$ the function~$r_z$ is a Herglotz function, whose properties are well known to be effective in the study of the spectral properties of $H_0$ and $H_1$ (e.g.~\cite{Don,Sim05}). 
Unfortunately, complications arise for more general~$V$, due in a large part to the fact that the functions~$r_z^j$ may have branching points. 

The functions $r_z^j$ admit other equivalent descriptions. 
As mentioned above, they are those values of~$s$ for which the {\em energy variable}~$z$ is an eigenvalue of~$H_s$.
They are also the poles of~\eqref{F: Rz(H_s)V} treated as a meromorphic function of the {\em coupling variable}~$s$. 
Indeed, the second resolvent identity implies
\[
 R_z(H_s)V = (1 + sR_z(H_0)V)^{-1}R_z(H_0)V
\]
and it follows from the analytic Fredholm alternative (see e.g.~\cite[Theorem~VI.14]{ReeSim1}) %~\cite[Theorem~1.8.2]{Yaf92}) 
that when considered as a function of~$s$, the poles of~\eqref{F: Rz(H_s)V} occur when $s=r_z^j$.
It is convenient to represent this relation symbolically as \[ z\lra s, \] 
in which the energy and coupling variables are in general multi-valued functions of one another which can have branching points.

Following the physics literature, poles of the scattering matrix $S(z; H_1,H_0)$ treated as a function of~$z$ are called {\em energy resonance points} (see e.g.~\cite{Zwo}). 
For this reason, we refer to those values of~$s$ such that $z\lra s$ as {\em coupling resonance points}, since they are poles of the scattering matrix $S(z; H_s,H_0)$ treated as a function of~$s$ (see~\cite{Aza11,Aza16,AzaDan18}). %, as discussed a little later.
Similarly, we refer to the functions~$r_z^j$ as {\em coupling resonance functions} of the pair~$H_0$ and~$V$.
%In fact the definition of energy resonance points requires that they belong to the `second sheet', which results upon analytic continuation through some interval of the essential spectrum. 
%Assuming this analytic continuation is possible, it is not difficult to show that a zero of the analytic continuation of some~$r_z^j$ must be an energy resonance point. 
%Regardless of the details of their connection, the point of mentioning energy resonances here is simply to provide some justification for our terminology. 
Because we will only be concerned with coupling resonances here, we usually leave the word `coupling' implicit. 

In the next few pages we discuss the connection between coupling resonances and the flow of singular spectrum and related topics, hoping to give a fuller perspective and further explain our motivation. However, this discussion digresses from the focus of this paper and the upcoming sections do not depend on it.

\medskip
Coupling resonances appear to contain a lot of information about the pair~$H_0$ and~$V$. 
For~instance, assuming $V$ satisfies a certain trace class condition, the Lifshitz-Krein spectral shift function (SSF) $\xi(\lambda;H_1,H_0)$ can be expressed in such terms (see~\cite{AzaDan19}):
\begin{equation}\label{F: SSF in terms of resonance data}
  \xi(\lambda; H_1,H_0) 
      = \frac 1{2\pi} \int_0^1 \sum_{j=1}^\infty \frac {2\beta_\lambda^j}{|r-\alpha_\lambda^j|^2 + |\beta_\lambda^j|^2}\,dr
    + \sum_{r_\lambda \in [0,1]} \rind(\lambda; H_{r_\lambda},V),
\end{equation}
where $r^j_{\lambda+i0} = \alpha_\lambda^{j} + i \beta_\lambda^{j}$ are the limits of resonance points with non-zero imaginary part~$\beta_\lambda^j$ and the second term is the sum of {\em resonance indices}, which are only non-zero for the discrete set of real limit points $r_{\lambda} = r^j_{\lambda+i0}$. 

The resonance index, which originates in~\cite{Aza16}, measures the flow of singular spectrum and provides the previously mentioned link between this phenomenon and the spectral properties of~\eqref{F: Rz(H_0)V}.
It is those eigenvalues of~\eqref{F: Rz(H_0)V} which converge to the real axis as $z=\lambda+iy\to\lambda+i0$ that are involved in the flow of singular spectrum of~$H_s$ through the point~$\lambda$. 
More specifically, if $N_+$ and $N_-$ respectively denote the numbers of eigenvalues of~\eqref{F: Rz(H_0)V} converging from the upper~$\mbC_+$ and lower~$\mbC_-$ complex half-planes to a real number~$-r_\lambda^{-1}$ as $y\to 0^+$, then their difference $N_+ - N_-$ is equal to the net flux of singular spectrum through~$\lambda$ as the perturbation~$H_s$ crosses the operator $H_{r_\lambda} = H_0 + r_{\lambda}V$.
The~infinitesimal singular spectral flow defined in this way is the {\em resonance index}, denoted 
\[
 \rind(\lambda;H_{r_\lambda},V) = N_+ - N_-.
\]

The sum of resonance indices at points~$r_{\lambda}$ from the compact interval~$[0,1]$, only finitely many of which can be non-zero, is called the {\em total resonance index} for the pair $H_0$ and~$H_1$.
For a point~$\lambda$ outside of the common essential spectrum~$\sigma_{ess}$ of $H_0$ and~$H_1$, the associated flow of singular spectrum is also known as {\em spectral flow} which has many characterisations.
In this case, for $\lambda\notin\sigma_{ess}$, direct proofs can be found in~\cite{Aza17} of the agreement of the total resonance index with four prominent definitions of spectral flow: in terms of the intersection number of eigenvalues with a point (\cite{APS}), the Fredholm index of a pair of projections~(\cite{ASS,Phi}), Robbin-Salamon's axioms (\cite{RobSal}), and the SSF~(\cite{Kre}). 
Of these definitions, the resonance index has the advantage of requiring minimal assumptions and moreover it is the only one capable of measuring the flow of singular spectrum within the essential spectrum as well. 

If~$\lambda$ belongs to the essential spectrum, it is convenient to assume a decomposition
\[ V = F^*JF, \] 
where $F$ is a closed $|H_0|^{1/2}$-compact operator and $J$ is a bounded self-adjoint operator. 
One choice that works for any $H_0$-compact self-adjoint $V$ is $F = |V|^{1/2}$ and $J = \sgn V$. 
This factorisation also provides a technique for dealing with more general symmetric forms~$V$ that are merely relatively form-compact with respect to~$H_0$.
Let the sandwiched resolvent be denoted \[ T_z(H_0) := FR_z(H_0)F^*. \]
Then instead of~\eqref{F: Rz(H_0)V} we can consider the compact operator 
\begin{equation}\label{F: Tz(H_0)J}
T_z(H_0)J,
\end{equation}
which shares the same non-zero eigenvalues, but has the advantage of possibly existing in the limit as~$z=\lambda+iy$ approaches the spectrum of~$H_0$. 
If the sandwiched resolvent has a norm limit 
\begin{equation*}\label{F: T_lambda+i0(H_0)}
 T_{\lambda+i0}(H_0) := \lim_{y\to 0^+}T_{\lambda+iy}(H_0),
\end{equation*} 
which for example, if~$V$ satisfies a relatively trace class condition, occurs for a.e.~point~$\lambda\in\mbR$ (see e.g.~\cite{BirEnt67,DeB62,Yaf92}), then the limits of eigenvalues of~\eqref{F: Rz(H_0)V} are eigenvalues of~\eqref{F: Tz(H_0)J} with $z=\lambda+i0$ and the resonance index is well-defined. 

The total resonance index is related to the Birman-Schwinger principle (see e.g.~\cite{Sim05,Sim79}), which is based on the relation $z \lra s$ and is used for example in the evaluation of bounds on the number of eigenvalues of Schr\"odinger operators. 
Supposing~$H_0$ is bounded from below, $V\leq 0$, and $z=\lambda < \inf\sigma(H_0)$, it states that the number of eigenvalues of~$H_1$ which are less than~$\lambda$ is equal to the number~$N_\lambda$ of eigenvalues of~$R_\lambda(H_0)V$ which are less than~$-1$.
In~\cite{Sob} (also see~\cite{BirPus}), it is observed that the Birman-Schwinger principle can be extended to a representation of the SSF for $\lambda\notin\sigma_{ess}$: again supposing $V\leq 0$, 
\begin{equation}\label{F: SSF = -N} 
  \xi(\lambda;H_1,H_0) = -N_\lambda. 
\end{equation}
Note that the total resonance index gives the same result, since for $V\leq 0$ all the eigenvalues of~\eqref{F: Rz(H_0)V} shift into $\mbC_-$ as $z=\lambda+i0 \to \lambda+iy$ with $y>0$.
The representation~\eqref{F: SSF = -N} of the flow of singular spectrum is extended to the essential spectrum in~\cite{Pus11}, where the focus is on an interval of essential spectrum in a neighbourhood of which the sandwiched resolvent~$T_z(H_0)$ is assumed to be uniformly continuous in the operator norm. 
%Conjecture: in this situation we are dealing with the flow of pure point spectrum.
The total resonance index provides a further extension, valid at any point~$\lambda$ where the limit $T_{\lambda+i0}(H_0)$ exists. 

Within the essential spectrum, the relationship of the resonance index to the SSF is clarified as follows. 
Assuming~$V$ is of relatively trace class type, $\xi(\lambda;H_1,H_0)$~and its absolutely continuous $\xi^{(a)}(\lambda;H_1,H_0)$ and singular $\xi^{(s)}(\lambda;H_1,H_0)$ parts can be respectively defined by the formula
\begin{equation*}\label{F: SSFs defn}
 \xi^{\#}(\varphi;H_1,H_0) = \int_0^1 \Tr\left( E_r^{\#}(\supp \varphi)V\varphi(H_r)\right)\,dr, \quad \varphi\in C_c(\mbR),
\end{equation*}
where the placeholder $\#$ is either removed altogether, or replaced by~$(a)$, or~$(s)$, in which case $E_r^{(a)}$ and $E_r^{(s)}$ denote the absolutely continuous and singular parts of the spectral measure~$E_r$ corresponding to the self-adjoint operator~$H_r = H_0 + rV$. 
The first case is the famous Birman-Solomyak formula for the SSF~\cite{BirSol}, while the variants defining its absolutely continuous and singular parts originate in~\cite{Aza11}.
To~be clear, in~any case this formula defines a locally finite measure via the Riesz representation theorem, which can be shown to be absolutely continuous and can therefore be identified with its density function (see e.g.~\cite{AzaDan18}).

The singular SSF is integer-valued and coincides with the total resonance index (see e.g.~\cite{Aza16,AzaDan18}): For a.e.~$\lambda\in\mbR$,
\begin{equation}\label{F: sSSF = total res ind}
 \xi^{(s)}(\lambda;H_1,H_0) = \sum_{r_\lambda\in [0,1]} \rind(\lambda;H_{r_\lambda},V).
\end{equation}
(Hence the remaining term on the right of~\eqref{F: SSF in terms of resonance data} is $\xi^{(a)}(\lambda;H_1,H_0)$.)
In fact the total resonance index was the result of a search for a more tangible representation of the singular SSF, after it became clear that, as a function of~$r\in\mbR$ for fixed~$\lambda$, $\xi^{(s)}(\lambda;H_r,H_0)$~is locally constant with discontinuities at the real limits~$r_\lambda$ of resonance functions. 
This fact was derived from a representation of the singular SSF in terms of the scattering matrix, which we now outline before returning to the focus of this paper.

\medskip
For our purposes, the (off-axis) scattering matrix can be defined by the `stationary formula'
\begin{equation}\label{F: stationary formula}
 S(z;H_r,H_0) := 1 - 2i\sqrt{\Im T_{z}(H_0)}J(1 + rT_z(H_0)J)^{-1}\sqrt{\Im T_z(H_0)},
\end{equation}
where $r\in\mbR$ and $z=\lambda+iy$, for $y\geq 0$.  
The case when $y=0$ is defined by the limit $y\to 0^+$, assuming it exists. 
This amounts to assuming the existence of~$T_{\lambda+i0}(H_0)$ along with $T_{\lambda+i0}(H_r)$, and for this to be possible, it must be that $s=0$ and $s=r$ are not limit points of resonance functions. 
Note that a resonance function cannot take real values except in the limit as $y\to 0$.

An algebraic calculation shows that the scattering matrix is unitary and 
thus belongs to the operator-normed group 
\begin{equation}\label{F: clU}
 %\clU := 
 \{\text{unitary operators } S \text{ such that } S-1 \text{ is compact}\}.
\end{equation}
The eigenvalues of any such operator lie on the unit circle~$\mbT$ and can only accumulate at the common essential spectrum~$\{1\}$.  
We now consider the limit of~\eqref{F: stationary formula} as $y\to +\infty$. Since in this case $\Im T_{z}(H_0) \to 0$, it follows that the scattering matrix $S(\lambda;H_1,H_0)$ for the pair~$H_0$ and~$H_1$ can be continuously connected to the identity operator~$1$ within the normed group \eqref{F: clU}.
In doing so, all of its eigenvalues are deformed along the circle to $1$. 
Following~\cite{Pus01}, we denote the net number of eigenvalues which cross a given point $e^{i\theta}\in\mbT$, or in other words the spectral flow through $e^{i\theta}$, by~$-\mu(\theta,\lambda;H_1,H_0)$. 
In this context, our preferred method of definition for the flow of discrete spectrum is based on the continuous enumeration of eigenvalues -- an intuitive but non-trivial result which is presented in \cite{ADT}. 
Assuming $V$ is relatively trace class, it is proved by A.~Pushnitski in~\cite{Pus01} that the average of~$-\mu(\theta,\lambda;H_1,H_0)$ over $e^{i\theta}\in\mbT$ is equal to the SSF. 
That is, for a.e.~$\lambda\in\mbR$,
\begin{equation}\label{F: SSF = mu}
 \xi(\lambda;H_1,H_0) = -\frac{1}{2\pi}\int_0^{2\pi} \mu(\theta,\lambda;H_1,H_0) \,d\theta.
\end{equation}

Spectral flow is a well-known homotopy invariant, so that $\mu(\theta,\lambda;H_1,H_0)$ is not altered by continuous deformations of the path connecting $S(\lambda;H_1,H_0)$ with the identity. 
However, there is a non-equivalent such path: instead of sending the imaginary part of the energy variable $z=\lambda+iy$ from $y=0$ to $y=+\infty$, send the coupling variable from $r=1$ to $r=0$. 
When defined by~\eqref{F: stationary formula}, it is evident (from the analytic Fredholm theorem) that the scattering matrix depends meromorphically on the coupling variable~$r$ and that its poles can only occur at real resonance points. 
However these turn out to be removable singularities, since the scattering matrix is unitary and thus bounded away from the discrete set of such resonance points. 

The definition~\eqref{F: stationary formula} is properly justified by the following theorem (see e.g.~\cite{AzaDan18}, the proof is based on an argument originating in~\cite{Aza16}): 
Assume the limiting absorption principle holds, by which we mean that the set $\Lambda$, of values of $\lambda$ such that both limits~$T_{\lambda+i0}(H_0)$ and $T_{\lambda+i0}(H_r)$ exist, has full measure in~$\mbR$. Then there is a spectral representation of the absolutely continuous part of~$H_0$ as a direct integral over~$\Lambda$, on which the definition of $S(\lambda;H_r,H_0)$ by~\eqref{F: stationary formula} coincides with the classical definition of the scattering matrix in terms of the fibres of the product of wave operators.
Since the full set~$\Lambda$ is given explicitly, it becomes possible to consider the scattering matrix as a function of the coupling variable. 
Assuming $V$ satisfies a relatively trace class condition, the set $\Lambda$ is known to have full measure for any $r\in\mbR$ (see e.g.~\cite{BirEnt67,DeB62,Yaf92}).
Moreover, if $\lambda$ is such that~$T_{\lambda+i0}(H_0)$ exists, then $T_{\lambda+i0}(H_r)$ must also exist and hence $\lambda\in\Lambda$, as long as $r$ is not a real resonance point. 

After removing a finite number of singularities as $r=1$ is deformed to $r=0$, we obtain an analytic path connecting $S(\lambda;H_1,H_0)$ with the identity. 
Following~\cite{Aza16}, let the spectral flow along this path be denoted by $-\mu^{(a)}(\theta,\lambda;H_1,H_0)$. 
Assuming $V$ is relatively trace class, the average of $-\mu^{(a)}(\theta,\lambda;H_1,H_0)$ over $e^{i\theta}\in\mbT$ is shown in~\cite{AzaDan18}, following an argument originating in~\cite{Aza16}, to be equal to the absolutely continuous SSF; for a.e.~$\lambda\in\mbR$,
\begin{equation}\label{F: a.c.SSF = mua}
 \xi^{(a)}(\lambda;H_1,H_0) = -\frac{1}{2\pi}\int_0^{2\pi} \mu^{(a)}(\theta,\lambda;H_1,H_0) \,d\theta.
\end{equation}

A representation of the singular SSF is obtained by subtracting~\eqref{F: a.c.SSF = mua} from~\eqref{F: SSF = mu}. 
By the additivity of spectral flow, the so-called {\em singular $\mu$-invariant}, defined by 
\[
 \mu^{(s)}(\lambda;H_1,H_0) := \mu(\theta,\lambda;H_1,H_0) - \mu^{(a)}(\theta,\lambda;H_1,H_0),
\] 
is the spectral flow of the scattering matrix $S(z;H_r,H_0)$, $z=\lambda+iy$, as it traverses the loop based at the identity which results as $(y,r)$ goes from $(+\infty,1)$ to $(0,1)$ to $(0,0)$.  
As a (finite) sum of winding numbers of eigenvalues, the singular $\mu$-invariant does not depend on the angle~$\theta$ and we find that for a.e.~$\lambda\in\mbR$,
\begin{equation}\label{F: sSSF = singular mu inv}
 \xi^{(s)}(\lambda;H_1,H_0) = -\mu^{(s)}(\lambda;H_1,H_0).
\end{equation}

Given the previous two representations of the singular SSF~\eqref{F: sSSF = total res ind} and~\eqref{F: sSSF = singular mu inv}, it follows that the total resonance index and singular $\mu$-invariant are equal for relatively trace-class~$V$ and a.e.~$\lambda\in\mbR$. 
Note however that a relatively trace class condition is not required for the definition of either of these representations. 
Suppose~$\lambda$ is such that both norm limits~$T_{\lambda+i0}(H_0)$ and $T_{\lambda+i0}(H_1)$ exist. 
Then both the total resonance index and the singular $\mu$-invariant are well-defined at~$\lambda$ and moreover, as shown in~\cite{AzaDan20}, these integers are equal. 
This leads to the conjecture that, although the SSF does not exist in this context, the singular SSF does exist and is given by the same integer.
Much more information about the decomposition of the SSF and its connections to scattering theory can be found e.g.~in~\cite{Aza11,AzaDan18} and for more about the total resonance index including other descriptions see e.g.~\cite{Aza16,Aza17}. 
%The~total resonance index has other descriptions too -- as the sum of signatures of certain self-adjoint so-called resonance matrices, and as the so-called singular $\mu$-invariant which is the sum of winding numbers of eigenvalues of the scattering matrix as is deformed to and from the identity in two different ways (see~\cite{Aza16}). 

\medskip
Returning to the focus of this paper, things are much simpler outside of the essential spectrum, as is to be expected. 
In its complement in $\mbR$, the properties of resonance functions are studied in detail in~\cite{Aza17}. 
Here we extend this study to the complex plane, with the aim to be able to approach the essential spectrum and allow a further study of its perturbation by relatively compact operators.
Although a number of results have proofs similar to those of~\cite{Aza17}, there are some essential differences and for this reason we believe that a detailed presentation of this material is warranted. %, since (as mentioned) a study of the off-axis behaviour of coupling resonance points near essential spectrum is important...

In~terms of the relation $z\lra s$, shifting into the complex plane entails working with non-self-adjoint operators~$H_s$. 
In fact, we will consider perturbing the operator~$H_s$ in another relatively compact direction~$W$.
This introduces a few novel hurdles, since some arguments used in~\cite{Aza17} are specific to the self-adjoint case. 
The techniques presented in this paper are not only more general, but are also simpler in some aspects. 
However, these new elements of the method do not void the approach of~\cite{Aza17} which allows to prove deeper results for the self-adjoint case.

Although the eigenspaces of $H_s$ at $z$ and of~\eqref{F: Rz(H_0)V} at $-s^{-1}$ coincide, the root spaces do not.
In~the self-adjoint case of~$s\in\mbR$, the eigenvalue~$z$ is semisimple, i.e.~the root space and eigenspace of~$H_s$ coincide, and for $s\in\mbC\setminus\mbR$ this is more or less an assumption we will make.
On the other hand, the root space of~\eqref{F: Rz(H_0)V} can be strictly larger than its eigenspace even for real~$s.$ 
Root vectors are called {\em resonance vectors} of the triple $(z;H_0,V)$.
In Section~\ref{S: Rz(N_v)} we consider the Laurent coefficients of the resolvent $R_z(H_s)$ as a function of the coupling variable~$s$ and their connection with the projection onto the space of resonance vectors. 
In particular, it is shown that the latter is equal to the product of the residue of the resolvent at a resonance point with the perturbation. 

Section~\ref{S: semisimple} concerns the assumptions used in later sections, which in essence amount to the following: Given a path of eigenvectors~$\phi(s)$ corresponding to the eigenvalue~$z(s)$ of~$H_s$, we assume the existence of a non-orthogonal conjugate eigenpath $\phi^*(s)$ corresponding to the eigenvalue~$\bar z $ of~$H_{\bar s}$. 
For any given eigenvector $\phi$ corresponding to a fixed eigenvalue~$z$ of $H_s$, the existence of an eigenvector~$\phi^*$ of~$H_{\bar s}$ corresponding to~$\bar z$ which is not orthogonal to~$\phi$, is an equivalent condition to the semisimplicity of the eigenvalue~$z$. 
In addition to this, we require a certain stability under perturbations of the coupling variable. These conditions are automatically satisfied in the self-adjoint case.

As shown in~\cite{Aza17}, the structure of the space of resonance vectors is related to the behaviour of eigenvalues of~$H_s$ as they cross a real threshold value $z = \lambda$. 
%This is also true inside of the essential spectrum $\sigma_{ess}$, but more can be said if $\lambda\notin\sigma_{ess}.$
For~example, assume that an eigenvalue $\lambda$ of~$H_{r_\lambda}$, for $r_\lambda\in\mbR$, is simple. 
Then, provided that~$\lambda$ is not also an eigenvalue of~$H_0$, the dimension of the root space of $R_\lambda(H_0)V$ corresponding to the eigenvalue $- r_\lambda^{-1}$ tells us, in particular, whether the eigenvalue $z$ of $H_s$ stops instantaneously at~$\lambda$, that is, whether ${dz}/{ds}$ is zero at $s = r_\lambda,$ and if so, in which direction it goes after stopping. 
That is, for a simple eigenvalue~$z$ of~$H_s$, the dimension of the root space is equal to the order of~$s$ as a critical point of~$z$. 
This result is generalised to the non-self-adjoint case in Section~\ref{S: order of e.path}.
We define the order of an eigenpath for~$H_s$, establish various characterisations, and show that it exceeds 1 iff~$z$ is a branching point of the eigenvalue~$-s^{-1}$ of~\eqref{F: Rz(H_0)V}. 
In particular, we prove the following theorem, in which $\bfA_{z}(H_{s},V)$ denotes a nilpotent operator associated to the space of resonance vectors.

\begin{thm} \label{T: criteria for splitting}
Suppose $z\lra s$, that is, $-s^{-1}$ is an eigenvalue of~$R_z(H_0)V$, or equivalently $z$ is an isolated eigenvalue of~$H_s = H_0 + sV$. 
Consider~$z$ as a function of~$s$ in a neighbourhood of some point $s_0$. 
Let~$\phi(s)$ be an eigenvector function corresponding to~$z(s)$. 
Assume~$z_0:=z(s_0)$ is a simple eigenvalue. 
Then the following assertions are equivalent:  
\begin{enumerate}
  \item the eigenpath~$\phi(s)$ has order~$\geq 2$ at~$s_0,$ 
  \item $\scal{\phi^*,V\phi(s_0)} = 0$ 
     for some eigenvector~$\phi^*$ of $H_{\bar s_0}$ corresponding to the eigenvalue $\bar z_0$, such that $\scal{\phi^*,\phi(s_0)}\neq 0,$
  \item $z'(s_0) = 0,$
  \item $(H_{s_0} - z_0)\phi'(s_0) = - V\phi(s_0),$ 
  \item the vector~$\phi(s_0)$ belongs to the image of $\bfA_{z_0}(H_{s_0},V)$, 
  \item $\bfA_{z_0}(H_{s_0},V) \phi'(s_0) = \phi(s_0),$
  \item $z_0$ is a branching point of the inverse function $s(z).$
  \end{enumerate}
\end{thm}

The equivalence (iii) $\iff$ (vii) is of course a well-known property of holomorphic functions. 

The structure of the projection onto the space of resonance vectors is considered in more detail in Section~\ref{S: structure of Pz}. We establish its decomposition into a sum of projections corresponding to cycles of resonance points. This provides a natural Jordan basis of resonance vectors, in which we consider a Schmidt representation of the resonance projection.

Section~\ref{S: tangent high order} demonstrates a connection between order and tangency to the so-called resonance set. The resonance set is by definition the algebraic variety of operators, obtained from $H_{s_0}$ by relatively compact perturbations of the form~$vW$ with $W = W^*$ and $v\in\mbC$, which share a given eigenvalue $z_0$.
For a simple eigenvalue~$z_0$ of~$H_{s_0}$, it is shown that a path of operators of the form $H_{s_0} + vW$ is tangent to the resonance set if and only if it has a corresponding eigenpath of order $\geq 2$. 

Finally, Section~\ref{S: example} is a short remark on the isospectrality solutions to Lax's equation in light of these results.

\section{Preliminaries}\label{S: prelim}
%\section{Anti-holomorphic vector functions}
A $C^1$-function $f$ acting from an open subset~$G$ of~$\mbC$ to a Hilbert space $\clH$ 
is \emph{anti-holomorphic}, if the limit 
$
  \bar \partial_{s}f(s) := \lim_{h \to 0} {\bar h}^{-1}(f(s+h) - f(s))
$
exists at every $s \in G.$ 
If~$f \colon G \to \clH$ is anti-holomorphic and $g \colon G \to \clH$ is holomorphic, then the scalar product 
$
  \scal{f,g}
$
is holomorphic and 
$
  \partial_s\scal{f,g} = \scal{\bar \partial_{s} f,g} + 
  \scal{f,\partial_s g}. 
$
If $f(s)$ is anti-holomorphic in $G,$ then $f(\bar s)$ is holomorphic in $G^* = \{s \in \mbC \colon \bar s \in G\},$ 
which shows the similarity of properties of holomorphic and anti-holomorphic functions. 
For a~$\mbC$-valued holomorphic function $z(s),$ 
$\partial_s^j z(s) = 0$ if and only~if~$\bar\partial^j_{s} \bar z(s) = 0, \ j=1,2,\ldots$

For a closed operator $T$ on $\clH$ we denote by $\rho(T)=\{z\in\mbC \colon T-z \ \text{has bounded inverse}\}$ the resolvent set of~$T,$ and for the resolvent we use the notation  
$
  R_z(T) = (T - z)^{-1}.
$
By $\sigma(T) = \mbC\setminus\rho(T)$ we denote the spectrum of~$T,$ 
and by $\sigma_d(T)$ the discrete spectrum of~$T.$ 
That is, $z \in \sigma_d(T)$ iff $z \in \sigma(T)$ and $T-z$ is Fredholm, the latter means $\im(T-z)$ is closed and both $\ker(T-z), \coker(T-z)$ are finite-dimensional. 
For~$z \in \sigma_d(T),$ the \emph{geometric} and \emph{algebraic multiplicities} of~$z$ are the dimensions of, respectively, the \emph{eigenspace} $\clV(z,T) := \ker (T-z)$ and the \emph{root space} $\bigcup_{k=1}^\infty \ker (T-z)^k.$ 
Elements of the root space are called \emph{root vectors} or generalised eigenvectors.
If the operator $T$ is normal, then the root space and eigenspace coincide.
The~essential spectrum of $T$ is $\sigma_{ess}(T):=\sigma(T) \setminus \sigma_d(T).$ 
Other definitions of essential spectrum exist and this one is the largest of all,
but for the operators considered here -- relatively compact perturbations of self-adjoint operators, all of these definitions coincide. 

If a family of closed operators $T(v)$ is holomorphic at $v=0$ in the general sense that $v\mapsto R_z(T(v))$ is bounded-holomorphic for some $z\in\rho(T(0))$, then in a neighbourhood of any finite subset of~$\sigma_d(T_0)$ there are, counting multiplicities, a constant finite number of eigenvalues of~$T(v)$ for all small enough~$v$. 
The number of distinct eigenvalues is also constant except for a discrete set of exceptional points.
Away from such exceptional points the eigenvalues are holomorphic functions of~$v$, which constitute the branches of one or more multivalued analytic functions. 
These functions can have only algebraic singularities and can therefore be represented by branches of one or several Puiseux series.
The same can be said of the corresponding eigenprojections, eigennilpotents, and therefore also eigenvectors.

\subsection{The affine space \texorpdfstring{$\clA$}{A}}
Throughout this paper~$H_0$ will denote a self-adjoint operator on a complex separable Hilbert space $\clH.$
We denote by $\clA_0$ the real vector space of all bounded self-adjoint $H_0$-compact operators~$V,$
the latter means that for some, and thus for any,~$z$ from the resolvent set $\rho(H_0),$ the operator~\eqref{F: Rz(H_0)V} is compact.
%\begin{equation} \label{F: Rz(H0)V is compact}
%  \text{the operator} \ R_z(H_0) V \ \text{is compact},
%\end{equation}

The results of this paper hold for certain unbounded perturbations too, but for simplicity we focus on bounded perturbations.
Our strategy will be to deal with unbounded perturbations at the end of sections in various brief remarks, whose purpose is to describe the most significant necessary adjustments.

%For a complex number~$s$ we denote by~$H_s$ the operator $H_0 + s V.$
By~$\mbC\clA_0$ we denote the complexification of $\clA_0.$
By Weyl's theorem (see e.g.~\cite{ReeSim4}), the operators from
\[
 \clA:=H_0+\mbC\clA_0 
\]
have the same essential spectrum, denoted $\sigma_{ess}(\clA).$ 
%(These operators are well-defined in the unbounded case as a consequence of the KLMN theorem.)
For a fixed perturbation $V \in \clA_0,$ we write $H_s = H_0 +sV,$ for $s \in \mbC.$ 
Isolated eigenvalues of~$H_s$ we denote by~$z_\tau(s),$ $\tau=1,2,\ldots.$
%We will use the letter $\tau$ as a subscript, $z_\tau(s),$ to enumerate these eigenvalues. 
If we work with a fixed eigenvalue function of~$H_s,$ we may write~$z(s).$
A corresponding eigenvector we denote by~$\phi(s):$
\begin{equation*} \label{F: H(s)chi(s)=z(s)chi(s)}
  H_s \phi(s) = z(s) \phi(s).
\end{equation*} 
% The eigenvalue functions $z_\tau(s)$ define a Riemannian surface, denoted~$\euM,$
% and its boundary (or more precisely its projection into~$\mbC$) 
% $\partial \euM$ we call the \emph{critical set}. 

If $z \in \sigma_d(H_s),$
we say that~$s$ is a \emph{resonance point} corresponding to~$z.$
%if~$z$ is an isolated eigenvalue of~$H_s.$ 
In what follows we shall often consider a particular value of the complex variable $s$ denoted~$s_0,$
and the corresponding value of $z(s)$ we denote~$z_0.$ 
We will be perturbing the operator $H_{s_0}$ in another direction $W \in \clA_0.$
In such a case we use $N_0$, instead of $H_{s_0}$, for an element of~$\clA$ and write 
\begin{equation*} %\label{F: Nv}
  N_v = N_0 + v W.
\end{equation*}
Thus, $N_0$ is a not necessarily self-adjoint relatively compact perturbation of $H_0.$
By the second resolvent identity, any $W \in \clA_0$ is $N_0$-compact.
An eigenvalue function of $N_v$ we denote $z_\tau(v)$ or $z(v),$ 
and a corresponding eigenvector function by~$\phi(v):$ 
\begin{equation*} \label{F: N(v)phi(v)=z(v)phi(v)}
  N_v \phi(v) = z(v) \phi(v).
\end{equation*} 

The variables~$z$ and $v$ from the relation $z \in \sigma_d(N_v)$
are multi-valued functions of each other which can have branching points.
%We agree to use the word ``branching'' for branching points of~$v$ as a function of~$z,$
%but to use the word ``splitting'' for branching points of~$z$ as a function of~$v.$
%Thus, we may say ``$z_0$ is a branching point'' or ``$s_0$ is a splitting point'' but we do not interchange this terminology.
%We emphasize that this usage of the word ``splitting'' is not traditional; usually,
%this word is used to describe a situation in which an eigenvalue splits into two or more eigenvalues when perturbed, but this may occur without branching. 

We note that splitting -- the situation in which a single eigenvalue splits into two or more distinct eigenvalues when perturbed, may occur with or without branching. 
The eigenvalues which arise from the splitting of a given eigenvalue~$z_0$ are said to belong to the $z_0$-group.

\begin{remark}
When it comes to unbounded perturbations, we will assume that the initial operator $H_0$ is semibounded. 
In this case if a symmetric perturbation~$V$ is relatively form-compact with respect to~$H_0,$ we may consider a decomposition of the form $V= F^*JF,$ where~$J$ is a bounded self-adjoint operator on an auxiliary Hilbert space~$\clK$ and $F\colon\clH\to\clK$ is a fixed operator which is $|H_0|^{1/2}$-compact (this can be taken as the definition of relative form-compactness, cf.~\cite[pp.~662--663]{Sim15}). 
We~will assume in addition that the operator~$F$ is closed.
It~is also convenient to choose~$F$ to have trivial kernel and cokernel, which can be achieved without loss of generality.
%We note that $F$ restricted to the domain of $|H_0|^{1/2}$ is a closable operator if and only if this domain has a dense intersection with the range of the compact operator $(F(|H_0| + 1)^{-1/2})^*.$
%More information about this setting can be found in~\cite{Dan}.

In this situation, we set $\clA_0 = \{V = F^*JF : J \text{ is bounded and self-adjoint} \},$ where $V$ is interpreted as the symmetric form $(f,g) \mapsto V[f,g] := \scal{Ff,JFg}$ defined on $\dom F.$
We define $T_z(H_0)$ to be the compact operator which is the closure of the product $FR_z(H_0)F^*$ and we study the spectral properties of the compact operator~\eqref{F: Tz(H_0)J} in place of~\eqref{F: Rz(H_0)V}. 

In case~$V$ is a $H_0$-compact bounded self-adjoint operator, the operators~\eqref{F: Rz(H_0)V} and~\eqref{F: Tz(H_0)J} share the same non-zero eigenvalues. 
Moreover, the eigenspaces satisfy 
\[ \clV(\sigma,T_z(H_0)J) = F\clV(\sigma,R_z(H_0)V) \] 
and the root spaces are related similarly.
%The operator~$T_z(H_0)J$ is known as the Birman-Schwinger operator, especially in the context of Schr\"odinger operators.

An operator $H_s = H_0 + sV$ from the affine space $\clA$ is well-defined as a form-sum in the sense of the KLMN Theorem, rather than an operator-sum as in the Kato-Rellich Theorem.
Note that~$H_0$ corresponds to a semibounded form (see e.g.~\cite[Theorem~VI-2.7]{Kat80}) and each form-sum~$H_s$ is sectorial (\cite[Theorem~VI-1.33]{Kat80}) and hence corresponds to a sectorial operator.
It~is no longer true, as it is in the case of bounded perturbations, that the operators of~$\clA$ share a common domain, however this is true of their corresponding sectorial forms. 
In~remarks concerning unbounded perturbations to follow, we will use square brackets to distinguish a sectorial form $N_0[\cdot\,,\cdot]$ from its corresponding sectorial operator~$N_0\in\clA$    
and the common form-domain of operators from $\clA$ will be denoted $\dom [\clA].$ 

%A holomorphic family $N(v)$ of operators from~$\clA$ is of type~(B) in the sense of~\cite{Kat80}.
\end{remark}

\subsection{Vector spaces \texorpdfstring{$\Upsilon^k_{z_0}(N_0,W)$}{Upsilon}}\label{SS: Upsilon}
Let $N_0 \in \clA,$ $W\in\clA_0,$ and $N_v = N_0 + vW.$
If $z_0 \in \sigma_{d}(N_0),$ then %by an easy calculation 
for any $v\in \mbC$ such that $z_0 \in \rho(N_v),$ the number 
%\begin{equation*} \label{F: v(-1)}
$ v^{-1} $
%\end{equation*} 
is an eigenvalue of 
\begin{equation} \label{F: Rz(Nv)W}
  R_{z_0}(N_v)W,
\end{equation} 
and for any such~$v$ %{F: H(s)chi(s)=z(s)chi(s)}
\begin{equation}\label{F: e.v. = order 1}
  \clV(z_0,N_0) = \clV(v^{-1}, R_{z_0}(N_v)W).
\end{equation}
% and the vector~$\phi(s)$ from \eqref{F: H(s)chi(s)=z(s)chi(s)} with $s = s_0$ 
% is also a corresponding eigenvector of \eqref{F: Rz(Nv)W}. 
In other words, for any $v$ such that $z_0 \in \rho (N_v),$
a vector $\chi \in \clV(z_0,N_0)$ is a solution of the equation 
\begin{equation} \label{F: [1-vRW]k phi=0}
  \big[1 - v R_{z_0}(N_v)W\big]^k\chi = 0
\end{equation} 
with $k=1.$ 
Since the compact operator \eqref{F: Rz(Nv)W} is not necessarily normal, its root space corresponding to the eigenvalue $v^{-1}$ can be larger than the eigenspace.
The root space consists of~$\chi \in \clH$ which obey \eqref{F: [1-vRW]k phi=0} for some~$k\geq 1.$ 
For any $k \geq 1$ \eqref{F: [1-vRW]k phi=0} holds for any $v$ if it holds for some~$v$ with $z_0 \in \rho(N_v).$
A vector~$\chi$ which obeys \eqref{F: [1-vRW]k phi=0} we call a \emph{resonance vector} of the triple $(z_0,N_0,W),$ and the smallest of positive integers~$k$ obeying \eqref{F: [1-vRW]k phi=0} is the \emph{order} of~$\chi.$ 

For $z_0 \in \sigma_d(N_0),$ the vector space of resonance vectors 
of order~$\leq k$ we denote
$
  \Upsilon^k_{z_0}(N_0,W).
$
The \emph{order} $d$ of the triple $(z_0,N_0,W)$ is the smallest of positive integers $d$ such that 
\[
  \Upsilon^{d+1}_{z_0}(N_0,W) = \Upsilon^{d}_{z_0}(N_0,W).
\]
Since $v^{-1}$ is an eigenvalue of a compact operator, such a number $d$ exists. 
We write 
$
  \Upsilon_{z_0}(N_0,W) 
$
for $\Upsilon^{d}_{z_0}(N_0,W).$
From the fact that a linear operator and its adjoint have kernels of the same dimension, it follows that for any $k$ (c.f.~\cite[Lemma 3.1.4]{Aza16}),
\[
  \dim \Upsilon^k_{z_0}(N_0,W) = \dim \Upsilon^k_{\bar z_0}(N_0^*,W).
\]
If we study a particular triple $(z_0,N_0,W)$ then we use notation 
$
  n := \dim \Upsilon_{z_0}(N_0,W)
$
and 
$
  m := \dim \Upsilon^1_{z_0}(N_0,W). % = \dim \clV(z_0,N_0).
$
So,~$m$ is the geometric multiplicity of the eigenvalue~$z_0$ of $N_0$. 
However $n$ is not its algebraic multiplicity, although it is that of the eigenvalue~$v^{-1}$ of~\eqref{F: Rz(Nv)W}.

\begin{remark}\label{R: 1st unbdd pf}
In the situation for unbounded perturbations, every occurrence of $R_{z_0}(N_v)$ and~$W$ should be replaced by $T_{z_0}(N_v)$ and $J,$ where $W = F^*JF.$ 
This is a general rule when it comes to treating the unbounded case and for many things, including the upcoming section, it is the only difference.

Instead of~\eqref{F: e.v. = order 1}, it happens that 
\begin{equation}\label{F: e.v. = order 1 (unbdd)}
 F\clV(z_0,N_0) = \clV(v^{-1}, T_{z_0}(N_v)J).
\end{equation}
Otherwise the above definitions are unchanged, e.g.~resonance vectors of $(z_0,N_0,W)$,
which in this case belong to the auxiliary space~$\clK,$ 
are by definition root vectors of $T_{z_0}(N_v)J$.

For convenience we include a short proof of~\eqref{F: e.v. = order 1 (unbdd)}:
For $z_0 \in \rho(N_v)$ such that $z_0 \in \sigma_d(N_0),$ suppose that $N_0\chi = z_0\chi.$ 
Then 
\[
 (N_v - z_0)[f,\chi] = vW[f,\chi] = v\scal{Ff,JF\chi}, \quad \forall f \in \dom [\clA].
\]
%where $(f,g)\mapsto N_v[f,g]$ denotes the sesquilinear form corresponding to the sectorial operator~$N_v,$ and~$\dom[\clA]$ denotes the common form-domain of operators in~$\clA.$ 
In particular, for any $f = R_{\bar z_0}(N_v^*)g\in\dom N_v^*$ (which we note is a core of the form $N_v^*[\cdot\,,\cdot]$), we find
\[
 \scal{g,\chi} = v\scal{FR_{\bar z_0}(N_v^*)g,JF\chi}, \quad \forall g \in \clH.
\]
Thus $\chi = v(FR_{\bar z_0}(N_v^*))^*JF\chi$ and hence
\[
 F\chi = vF(FR_{\bar z_0}(N_v^*))^*JF\chi = vT_{z_0}(N_v)JF\chi.
\]
Each of the above implications is reversible, completing the proof.
\end{remark}

\section{Laurent coefficients of \texorpdfstring{$R_{z_0}(N_v)$}{Rz0(Nv)}}\label{S: Rz(N_v)}
Most of the material of this section can be found in~\cite[Section 3]{Aza16},
but this presentation has some new elements and is shorter. 
We consider the resolvent~$R_{z_0}(N_v)$ as a meromorphic function of~$v$. 
That this is possible can be seen as a result of the analytic Fredholm alternative (see e.g.~\cite[Theorem~VI.14]{ReeSim1}) %\cite[Theorem~1.8.2]{Yaf92}) 
and the resolvent identity
\begin{equation}\label{F: res id}
 R_{z_0}(N_v) = (1 + (v-u)R_{z_0}(N_u)W)^{-1}R_{z_0}(N_u).
\end{equation}
% If $m=1,$ then in a neighbourhood of splitting~$s_0$ 
% $$
%   z(s) = z_0 + c_{\frac i\delta} (s-s_0)^{\frac i\delta} + c_{\frac{i+1}\delta} (s-s_0)^{\frac{i+1}\delta} + \ldots
% $$
\begin{lemma} \label{L: orders equal}
Let $z_0 \in \sigma_d(N_0).$ 
%Order of the pole $v = 0$ of the meromorphic function $v \mapsto R_{z_0}(N_v)$ is~$d.$
The meromorphic functions $R_{z_0}(N_v)$ and $R_{z_0}(N_v)W$ have a pole of the same order at $v=0.$ 
\end{lemma}
\begin{proof} %The order $\geq d$ since $R_{z_0}(N_v)W$ has order $d$ by definition. 
%This also implies that the order $\leq d$ since 
The question is why the order of the pole $v=0$ of $R_{z_0}(N_v)$ cannot decrease after multiplication by~$W.$
This follows from the second resolvent identity 
\begin{equation*} \label{F: R(v)=R(u)+(u-v)R(v)WR(u)}
  R_{z_0}(N_v) = R_{z_0}(N_u) + (u-v)R_{z_0}(N_v)W R_{z_0}(N_u). 
\end{equation*}
\end{proof}
% Since the order of the pole of $v=0$ for $R_{z_0}(N_v)W$ we denote $d,$
% on the basis of this lemma we can write

Here and below, we assume that $z_0\in\sigma_d(N_0).$ 
Let 
\begin{equation*} \label{F: Laurent for Rz0(Nv)}
  \begin{split}
  R_{z_0}(N_v) & = \sum_{j=-d}^\infty v^j K_{j+1}(z_0,N_0,W)  \\
               & = v^{-d} K_{-d+1}  + \ldots + v^{-2} K_{-1} + v^{-1} K_0 + K_1 + v K_2 + \ldots
  \end{split}
\end{equation*}
be the Laurent expansion, so 
\begin{equation} \label{F: Kj = oint R vj}
  K_j(z_0,N_0,W) = \frac 1{2\pi i} \oint _{C(0)} R_{z_0}(N_v)v^{-j}\,dv,
\end{equation}
and $K_{-d+1}$ is non-zero. 
The number $d$ appearing here is equal to the order of the triple $(z_0,N_0,W)$ by Lemma~\eqref{L: orders equal}.

%From \eqref{F: bfA(j)=} and \eqref{F: Kj = oint R vj} we also have for $k=0,1,\ldots$ 
% \begin{equation} \label{F: Ak=K(-k)W and Bk=WK(-k)}
%   \bfA_{z_0}^k = K_{-k}W \quad \text{and} \quad \bfB_{z_0}^k = W K_{-k}.
% \end{equation}
\begin{lemma} \label{L: K(-k)WK(-j)=K(-k-j)}
For $k,j\geq 0,$ 
\begin{equation*} \label{F: K(-k)WK(-j)=K(-k-j)}
  K_{-k}W K_{-j} = K_{-k-j}.
\end{equation*}
\end{lemma}
\begin{proof}
Using~\eqref{F: Kj = oint R vj} we have 
\[
  (E) := (2\pi i)^2 K_{-k}W K_{-j} = \oint_{C_u(0)}\oint_{C_v(0)} R_{z_0}(N_v)W R_{z_0}(N_u) v^{k}u^{j}\,dv\,du,
\]
where $C_u(0)$ and $C_v(0)$ are contours enclosing zero. 
We choose $C_v(0)$ to be strictly inside $C_u(0).$ 
Applying the second resolvent identity gives
\[
  (E) = \oint_{C_u(0)}\oint_{C_v(0)} \big[R_{z_0}(N_v)- R_{z_0}(N_u)\big] (u-v)^{-1}v^{k}u^{j}\,dv\,du.
\]
With the choice of the contours, the summand 
which contains the factor $R_{z_0}(N_u)$ is a holomorphic function of $v$ on and inside $C_v(0),$
and thus its integral vanishes. Hence, Cauchy's theorem gives 
\begin{align*}
      (E) & = \oint_{C_v(0)}R_{z_0}(N_v)v^{k} \oint_{C_u(0)}  (u-v)^{-1}u^{j}\,du\,dv \\
          & = 2\pi i \oint_{C_v(0)}R_{z_0}(N_v)v^{k} v^j \,dv = (2\pi i)^2 K_{-k-j}. 
\end{align*} 
\end{proof}

Let 
\begin{equation} \label{F: Pz=K0*W, Qz=W*K0}
  P_{z_0}(N_0,W) = K_0(z_0,N_0,W)W, \quad Q_{z_0}(N_0,W) = W K_0(z_0,N_0,W),
\end{equation}
and for $k>0$
\begin{equation} \label{F: Ak=K(-k)W and Bk=WK(-k)}
  \bfA^k_{z_0}(N_0,W) = K_{-k}(z_0,N_0,W)W, \quad \bfB^k_{z_0}(N_0,W) = W K_{-k}(z_0,N_0,W).
%  \bfA_{z_0}^k = K_{-k}W \quad \text{and} \quad \bfB_{z_0}^k = W K_{-k}.
\end{equation}
% $$
%   \bfA^k_{z_0}(N_0,W) = K_{-k}(z_0,N_0,W)W, \quad \bfB_{z_0}(N_0,W) = W K_{-k}(z_0,N_0,W),
% $$
Lemma~\ref{L: K(-k)WK(-j)=K(-k-j)} implies that $P_{z_0}(N_0,W)$
and $Q_{z_0}(N_0,W)$ are idempotents, and,
assuming the convention that $\bfA^0_{z_0} = P_{z_0},$ 
$\bfB^0_{z_0} = Q_{z_0},$
that for $k,j\geq 0$ we have $\bfA^k_{z_0}\bfA^j_{z_0} = \bfA^{k+j}_{z_0}$ 
and $\bfB^k_{z_0}\bfB^j_{z_0} = \bfB^{k+j}_{z_0},$
which justifies the notation.

From the definition \eqref{F: Kj = oint R vj} of $K_j$ it follows that 
\[
  (K_j(z_0,N_0,W))^* = K_j(\bar z_0,N_0^*,W). 
\]
This implies the equalities 
\begin{equation} \label{F: P*=Q and A*=B}
  (P_{z_0}(N_0,W))^* = Q_{\bar z_0}(N_0^*,W), \quad
  (\bfA_{z_0}(N_0,W))^* = \bfB_{\bar z_0}(N_0^*,W),
\end{equation} 
and the following equality which will be frequently used without reference: for $f, g \in \clH$,
\[
  \scal{\bfA_{\bar z_0}(N^*_0,W)f,W g} = \scal{f,W\bfA_{z_0}(N_0,W)g}. 
\]
By Lemma~\ref{L: K(-k)WK(-j)=K(-k-j)}, for $k\geq 0$ 
\begin{equation} \label{F: K(-k)=K0WK(-k)}
  \begin{split}
      K_{-k} = \bfA_{z_0}^k K_0 = K_0 \bfB_{z_0}^k. 
  \end{split}
\end{equation}

\bigskip
If $j<0$, then $\oint_{C_u(0)}  (u-v)^{-1}u^{j}\,du = 0$
for $v$ inside $C_u(0),$ because the residues sum to zero in this case, and so the argument from the proof of Lemma~\ref{L: K(-k)WK(-j)=K(-k-j)} yields
\begin{lemma} \label{L: K(k)WK(m)=0} If one of $k$ or $j$ is $\geq 0$ and the other $<0,$ then 
$K_{-k}WK_{-j} = 0.$
\end{lemma}

\begin{lemma} \label{L: K(k+j)=-K(k)WK(m) if k,j>0} If $k,j>0$ then $K_kWK_j = - K_{k+j}.$
\end{lemma}
\begin{proof} The proof is the same as that of Lemma~\ref{L: K(-k)WK(-j)=K(-k-j)},
with the following changes: choose $C_v(0)$ to be strictly outside $C_u(0),$ 
and apply Cauchy's theorem to the exterior of $C_u(0),$
including~$\infty,$ where $u^{-j}$ is holomorphic. Since 
$C_u(0)$ is anticlockwise oriented, this also gives the sign. 
\end{proof}

%If at least one of $k$ or $j$ is negative, the equality \eqref{F: K(-k)WK(-j)=K(-k-j)} fails, but 
% Multiplying both sides of \eqref{F: R(v)=R(u)+(u-v)R(v)WR(u)} by $v^{k-1} u^{j-1}, \ j,k \in \mbZ,$ and taking
% the double contour integral at $v = 0$ and $u = 0$ gives
% \begin{equation} \label{F: K=K+KWK-KWK}
%   \delta_{j,0} K_{-k+1} = \delta_{k,0} K_{-j+1} + K_{-k+1}W K_{-j} - K_{-k}W K_{-j+1}, \ j,k \in \mbZ.
% \end{equation}
% Letting here $j=0$ gives 
% \begin{equation*} %\label{F: K=KWK-KWK}
%   \begin{split}
%   K_{-k+1} & = K_{-k+1}W K_0 - K_{-k}W K_1 \\
%            & = K_0 W K_{-k+1} - K_1 W K_{-k}  \quad (k\neq 0).
%   \end{split}
% \end{equation*}
% If $k>0,$ then by \eqref{F: K(-k)WK(-j)=K(-k-j)} this implies 
% $$
%   K_{-k}W K_1 = K_1 W K_{-k}.
% $$
% The equalities \eqref{F: K=K+KWK-KWK} imply the second resolvent identity \eqref{F: R(v)=R(u)+(u-v)R(v)WR(u)}, 
% so they are equivalent. 
% Choosing $k=d$ and $j\neq 0$ gives (since $d \neq 0$ and $K_{-d}=0$) 
% $$
%   0 = K_{-d+1} W K_{-j}, \ \ j\neq 0.
% $$

\begin{lemma} For $k \geq 0$ we have $\im K_{-k} = \im \bfA^{k}_{z_0}.$ In particular, $\im K_0 = \im P_{z_0}.$
\end{lemma}
\begin{proof} 
\eqref{F: Ak=K(-k)W and Bk=WK(-k)} gives $\im \bfA^{k}_{z_0}  \subset \im K_{-k}$
and \eqref{F: K(-k)=K0WK(-k)} gives $\im K_{-k} \subset \im \bfA^{k}_{z_0}.$
%Note that $\Upsilon_{z_0}(N_0) = \im P_{z_0}$ follows from the definitions.
\end{proof}

%The next lemma is reproduced from \cite[Proposition 3.2.3]{Aza16}.
%\begin{lemma}
%The image of the idempotent $P_{z_0}$ is the space of resonance vectors, that is,
%$\Upsilon_{z_0}(N_0) = \im P_{z_0}$.
%\end{lemma}
%\begin{proof}
%By definition we have
%\[
% P_{z_0}(N_0,W) = \frac{1}{2\pi i} \oint_{C(0)} R_{z_0}(N_v)W \,dv,
%\]
%where $C(0)$ is a small counter-clockwise oriented circle around~$0$. 
%Let~$u$ denote a fixed point outside of $C(0)$ for which $z_0\in\rho(N_u)$. 
%Then by the resolvent identity~\eqref{F: res id}, we find
%\begin{align*}
% P_{z_0}(N_0,W) &= \frac{1}{2\pi i} \oint_{C(0)} (1 + (v-u)R_{z_0}(N_u)W)^{-1}R_{z_0}(N_u)W \,dv
% \\&= \frac{1}{2\pi i} \oint_{C(0)} (v-u)^{-1}(1 - (1 + (v-u)R_{z_0}(N_u)W)^{-1}) \,dv.
%\end{align*}
%By the choice of~$u$, the function $(v-u)^{-1}$ is analytic within $C(0)$ so its integral vanishes. 
%Moreover, as $v$ traverses $C(0)$, the variable $\sigma = (u - v)^{-1}$ makes a counter-clockwise oriented closed loop $C(u^{-1})$ around $u^{-1}$. 
%Thus by substituting the variable $\sigma$,
%\begin{align*}
% P_{z_0}(N_0,W) &= \frac{1}{2\pi i} \oint_{C(0)} (u-v)^{-1}(1 + (v-u)R_{z_0}(N_u)W)^{-1} \,dv
% \\&= \frac{1}{2\pi i} \oint_{C(u^{-1})} \sigma(1 - \sigma^{-1}R_{z_0}(N_u)W)^{-1} \sigma^{-2} \,d\sigma
% \\&= \frac{1}{2\pi i} \oint_{C(u^{-1})} (\sigma - R_{z_0}(N_u)W)^{-1} \,d\sigma,
%\end{align*}
%which is the Riesz idempotent for $R_z(N_u)W$ at the eigenvalue $u^{-1}$. 
%\end{proof}

Based on previous material, we can write the Laurent expansion of $R_{z_0}(N_v)$ as 
\begin{equation*}\label{F: Laurent expansion of R_z(N_v)}
  \begin{split}
      R_{z_0}(N_v) & = v^{-d}\bfA^{d-1}_{z_0}K_0 + \ldots  v^{-2}\bfA_{z_0}K_0 + v^{-1} K_0 \\
              & \mbox{ } \hskip 2 cm + K_1 - v \tilde A K_1 + v^2 \tilde A^2 K_1 - \ldots,
  \end{split}
\end{equation*}
where $\tilde A = K_1 W.$

\bigskip

By definition, for $k \geq 1$
\begin{equation} \label{F: depth of res vector}
  \text{a resonance vector\ } \phi \text{\ has \emph{depth at least~$k$}, if\ } 
  \phi \in \im \bfA^k_{z_0}(N_0,W).
\end{equation}

The following facts come from results in \cite{Aza16}, namely Proposition~3.2.3 and Theorem~3.4.3, whose proofs apply here more or less verbatim.
The vector space of resonance vectors $\Upsilon_{z_0}(N_0,W)$ is $\im P_{z_0}(N_0,W).$
Also, for all $k=1,2,\ldots,d$ 
\begin{equation*}\label{F: A_zUpsilon_z^k = Upsilon_z^(k-1)}
  \bfA_{z_0}(N_0,W) \Upsilon^k_{z_0}(N_0,W) = \Upsilon^{k-1}_{z_0}(N_0,W),
\end{equation*}
where by agreement $\Upsilon^0_{z_0}(N_0,W) = \{0\}.$ Thus, the nilpotent operator
$\bfA_{z_0}(N_0,W)$ decreases the order of resonance vectors by~$1.$

\section{Semisimple eigenvalues and the conjugate of an eigenpath}\label{S: semisimple}
Recall that we write $z \lra s$ if $z$ is an eigenvalue of $H_s=H_0+sV,$ that is, $s$ is a resonance point corresponding to 
the triple $(z,H_0,V).$
As we shall see shortly, imposing a generic condition, with a resonance vector~$\phi$ corresponding to~$z\lra s,$ one can naturally associate a resonance vector $\phi^*$ corresponding to~$\bar z \lra \bar s.$ 
In order to explain the nature of this condition, assume that the geometric multiplicity of~$z$ as an eigenvalue of $H_s$ is~$1.$ 
In such a case, the geometric multiplicity of~$\bar z$ as an eigenvalue of $H_{\bar s}$ is also equal to~$1,$ 
and thus there are eigenvectors~$\phi$ and~$\phi^*$ corresponding to these eigenvalues, which are unique up to scaling. 
The condition is that the vectors~$\phi$ are~$\phi^*$ are not orthogonal.
In general, if the multiplicity of~$z$ is greater than~$1,$ this condition is to be replaced by non-orthogonality of $\phi$ and $\clV(\bar z, H_{\bar s}).$ 
As reviewed below, this property is equivalent to the semisimplicity of the eigenvalue $z$.
In addition, the upcoming sections require a certain kind of stability under perturbations of~$H_s$, which in essence amounts to the regular behaviour of both $\phi$ and $\phi^*$.
These assumptions are clarified in this section.

For $z_0 \in \sigma_d(N_0),$ with the pair $(z_0,N_0),$ 
apart from $P_{z_0}(N_0,W)$ defined in \eqref{F: Pz=K0*W, Qz=W*K0},
we can associate another idempotent, the standard eigenprojection (cf.~\cite[(II.1.16)]{Kat80})
\begin{equation} \label{F: E(z0,N0)=}
  E(z_0,N_0) := - \frac 1{2\pi i} \oint_{C(z_0)} (N_0 - \zeta)^{-1}\,d\zeta.
\end{equation}
There is also the associated nilpotent operator~\cite[(II.1.22)]{Kat80} 
\begin{equation}\label{F: D(z0,N0)=}
  D(z_0,N_0) := - \frac 1{2\pi i} \oint_{C(z_0)} (\zeta-z_0)(N_0 - \zeta)^{-1}\,d\zeta.
\end{equation}
We note 
\[
  (D(z_0,N_0))^* = D(\bar z_0,N^*_0), \quad (E(z_0,N_0))^* = E(\bar z_0,N^*_0).
\]

\begin{prop} \label{P: prop X} Let $z_0 \in \sigma_d(N_0).$
The following assertions are equivalent.
\begin{enumerate}
  \item \label{I: im E=V} $\im E(z_0,N_0) = \clV(z_0,N_0).$
  \item \label{I: non-zero phi not perp Upsilon bar} 
          Any non-zero vector~$\phi$ from~$\clV(z_0,N_0)$ is not orthogonal to some vector of $\clV(\bar z_0,N_0^*).$
  \item \label{I: proj of Ups-n has zero kernel} 
         The orthogonal projection of~$\clV(z_0,N_0)$ to $\clV(\bar z_0,N_0^*)$  
            is a linear isomorphism.
  \item \label{I: D = 0} $D(z_0,N_0) = 0.$
  \item \label{I: im bar E=V bar} $\im E(\bar z_0,N_0^*) = \clV(\bar z_0,N_0^*).$
  \item \label{I: non-zero bar phi not perp Upsilon} 
          Any non-zero vector~$\phi$ from~$\clV(\bar z_0,N_0^*)$ is not orthogonal to some vector of $\clV(z_0,N_0).$
  \item \label{I: proj of Ups-n bar has zero kernel} 
         The orthogonal projection of ~$\clV(\bar z_0,N_0^*)$ to $\clV(z_0,N_0)$  
            is a linear isomorphism.
  \item \label{I: bar D = 0} $D(\bar z_0,N_0^*) = 0.$
\end{enumerate}
\end{prop}
\begin{proof} The equivalences~\ref{I: bar D = 0} $\iff$~\ref{I: im bar E=V bar} 
$\iff$~\ref{I: im E=V} $\iff$~\ref{I: D = 0}, 
\ref{I: non-zero phi not perp Upsilon bar} $\iff$ 
\ref{I: proj of Ups-n has zero kernel} $\iff$~\ref{I: proj of Ups-n bar has zero kernel} 
$\iff$~\ref{I: non-zero bar phi not perp Upsilon} are obvious and/or well-known. 
The equivalence~\ref{I: im E=V} $\iff$~\ref{I: non-zero phi not perp Upsilon bar} is also obvious, but still we give its proof. 

\ref{I: im E=V} $\Rightarrow$~\ref{I: non-zero phi not perp Upsilon bar}.
For $0 \neq \phi \in \clV(z_0,N_0)$ we have 
$
  0 \neq \scal{\phi,\phi} = \scal{\phi,E(z_0,N_0)\phi} = \scal{E(\bar z_0,N_0^*)\phi,\phi}.
$
So, the non-zero vector $E(\bar z_0,N_0^*)\phi \in \clV(\bar z_0,N_0^*)$ is not orthogonal to~$\phi.$

\ref{I: non-zero phi not perp Upsilon bar} $\Rightarrow$~\ref{I: im E=V}.
If $\im E(z_0,N_0) \neq \clV(z_0,N_0),$ then for some non-zero $\phi \in \clV(z_0,N_0)$
and $\psi$ we have $D(z_0,N_0) \psi = \phi.$ So, for any $\phi^* \in \clV(\bar z_0,N_0^*)$
we have 
$
  \scal{\phi^*,\phi} = \scal{D(\bar z_0, N_0^*)\phi^*,\psi} = 0.
$
\end{proof}

\begin{defn} \label{D: semisimple point}
 $z_0 \in \sigma_d(N_0)$ is called \emph{semisimple,} if some and therefore any of the assertions of Proposition~\ref{P: prop X} hold. 
\end{defn}

There are yet other equivalent definitions of semisimplicity in the literature. Another worth noting is for the gap $\hat{\delta}(\mathsf{M},\mathsf{N})$ between the subspaces $\mathsf{M}=\clV(z_0,N_0)$ and $\mathsf{N}=\clV(\bar z_0, N_0^*)$, as defined in \cite[IV-\S2.1]{Kat80}, to be less than 1 -- its equivalence can be seen with items (ii) and (iv) of Proposition~\ref{P: prop X} taken together. 

Numerical experiments easily generate pairs of self-adjoint matrices $H_0$ and~$V$ such that some $z_0 \in \sigma(H_0+s_0 V)$ fails to be semisimple. %right?
%Indeed, if a pair $H_0$ and $V$ has property~$X$ for all~$z_0$ and $s_0$ then as $s$ goes along any loop the eigenvalues of $H_0+sV$ should return to their original positions, while in fact they often undergo a non-trivial permutation. 
If~$s_0$ is real, then all eigenvalues are semisimple,
%The requirement for eigenvalues to be semisimple -- I don't think we need this
which is one of the differences between the real case considered in~\cite{Aza17} and the present one. 

\bigskip
%The conjugate of $\phi(v)$, corresponding to $z(v)$, is $\phi^*(v) := P_{\bar z(v)}(N_v^*,W)\phi(v)$?
%\section{Semisimple eigenvalues of $N_0$ and the conjugate of an eigenpath}\label{S: conjugate}
For $N_0\in\clA$ and $z_0\in\sigma_d(N_0)$, we now reintroduce a perturbation $W\in\clA_0$. 
Let $z(v)$ be an eigenvalue function corresponding to $N_v = N_0 + vW$, with $z(0) = z_0$. 
Then~$z(v)$ is holomorphic in a deleted neighbourhood of $0$, but not necessarily at $v=0$ if it is an exceptional point, that is, if $v=0$ is one of the discrete set of points where the number of distinct eigenvalues of $N_v$ within a small neighbourhood of $z_0$ is different from its otherwise constant value.

We will restrict our attention to the case when~$z(v)$ is not subject to branching at $v=0$, in other words, we assume $z(v)$ is analytic at $v=0$. 
In this case (see e.g.~\cite[Chapter~18]{GLR}) there must exist an eigenvector function, or eigenpath, corresponding to $z(v)$ which is also analytic in a neighbourhood of $v=0$. 
Such an analytic eigenpath will be denoted by $\phi(v)$. 
The anti-holomorphic eigenvalue function of $N^*_v = N_0^* + \bar v W$, equal to $\overline{z(v)}$, will be denoted~$z^*(v).$ 
A~corresponding anti-holomorphic eigenpath will be denoted $\phi^*(v)$.
Given $\phi(v)$, our aim is to canonically assign a non-orthogonal $\phi^*(v)$.

Assume first that $z_0$ is a simple eigenvalue, i.e.~its algebraic multiplicity is equal to~$1.$ 
Then by Proposition~\ref{P: prop X}, $\phi(v)$ and $\phi^*(v)$ are not orthogonal and we can normalise 
$\phi^*(v)$ in such a way that 
\begin{equation*} \label{F: (phi*,phi)=1}
  \scal{\phi^*(v), \phi(v)} = 1.
\end{equation*}
We call the resulting eigenpath $\phi^*(v)$ %obeying \eqref{F: (phi*,phi)=1}
the \emph{conjugate} of $\phi(v).$ 
Clearly, the conjugate of an eigenpath of geometric multiplicity $1$ is unique. 

In the case of arbitrary geometric multiplicity~$m$, the following assumption is suggested by developments in~\cite{HryLan} and~\cite{LMZ}. 
Some terminology beforehand:
Let $z_0\in\sigma_d(N_0)$ and let $z(v)$ be an eigenvalue of $N_v$ which converges to $z_0$ (and is not necessarily analytic there). Suppose $\phi(v)$ is a continuous normalised eigenvector corresponding to $z(v)$. 
Then, following~\cite{HryLan}, $\phi(0)$~will be called a {\em generating eigenvector} of~$N_v$ corresponding to $z(v)$. 

A generating eigenvector is necessarily an eigenvector, but the converse is not true in general. 
The comparison gives some indication about the presence of branching at~$z_0$. 
Consider for example the following results taken from \cite{HryLan}:
If branching occurs at~$z_0$, then eigenvectors $\phi_\tau(v)$, $j=1,2,\ldots$ corresponding to different branches $z_\tau(v)$, $j=1,2,\ldots$ of the same Puiseux series can be chosen so that they determine the same generating eigenvector $\phi_1(0) = \phi_2(0) = \ldots$ (\cite[Lemma~3.2]{HryLan}).
On the other hand, if $z_0$ is semisimple and generating eigenvectors span the eigenspace $\clV(z_0,N_0)$, then there can be no branching at $z_0$ (\cite[Proposition 3.5]{HryLan}). 

\begin{assum}\label{A: weak conjugate}
Let $z(v)$ be an eigenvalue of $N_v$ converging to $z_0\in\sigma_d(N_0)$.
For every generating eigenvector $\phi$ of~$N_v$ corresponding to $z(v)$, there exists a generating eigenvector $\phi^*$ of $N_v^*$ corresponding to $z^*(v)$ which is not orthogonal to $\phi$.
\end{assum}

The next theorem, here for convenience, is a reproduction of~\cite[Theorem~3.6]{HryLan}.
\begin{thm}\label{T: assum}
If Assumption~\ref{A: weak conjugate} holds for an eigenvalue $z(v)$ of $N_v$, then $z(v)$ is analytic in a neighbourhood of $v=0$ and semisimple in a deleted neighbourhood of $v=0$.
Suppose the eigenvalue~$z_0$ has arbitrary (algebraic) multiplicity~$m$ and Assumption~\ref{A: weak conjugate} holds for all eigenvalues of the $z_0$-group. 
Then for all small enough~$v$, the $z_0$-group consists of~$m$ eigenvalues $z_\tau(v)$, $\tau = 1,2,\ldots,m$, with any repeated eigenvalues being semisimple. 
All eigenvalues and corresponding eigenvectors can be chosen to be analytic. In particular, $z_0$ is semisimple, so that its geometric multiplicity is~$m$. 
\end{thm}

\begin{proof}
The proof proceeds by contradiction and makes use of (\cite[Lemma 3.1]{HryLan}): 
For distinct eigenvalue functions $z_\tau(v)$ and $z_\mu(v)$ of~$N_v$, any continuous eigenvectors $\phi_\tau(v)$ and $\phi^*_\mu(v)$ corresponding respectively to the eigenvalues $z_\tau(v)$ of $N_v$ and $z^*_\mu(v)$ of $N_v^*$, must satisfy 
\begin{equation}\label{E: (phi*,phi) = 0}
 \scal{\phi^*_\mu(0),\phi_\tau(0)} = 0.
\end{equation}

Assume there is an eigenvalue $z_\tau(v)$ of~$N_v$ for which Assumption~\ref{A: weak conjugate} holds and yet $z_\tau(v)$ tends to~$z_0$ as $v\to 0$ and is not analytic. 
Then~$z_\tau(v)$ must be a branch of some Puiseux series. 
Let $z_{\tau'}(v)$ be any different branch of the same algebraic function. 
Let $\phi_\tau(v)$ and $\phi_{\tau'}(v)$ be corresponding continuous eigenvector functions, which it is possible to choose such that $\phi_\tau(0) = \phi_{\tau'}(0) =: \phi_0$ (see \cite[Lemma~3.2]{HryLan}).
Let $\phi_\mu^*(v)$ be any eigenvector corresponding to some eigenvalue $z_\mu^*(v)$ of~$N^*_v$ such that both $z_\mu^*(v)$ and $\phi_\mu^*(v)$ are continuous.
Since $z_\mu(v)$ can not be identical to both $z_\tau(v)$ and $z_{\tau'}(v)$, it follows from~\eqref{E: (phi*,phi) = 0} that
\[
 \scal{\phi_\mu^*(0),\phi_0} = 0,
\]  
which is a contradiction.
Hence~$z_\tau(v)$ must be analytic.

Suppose $z_\tau(v)$ is not semisimple for $v=v_j$, $j\in\mbN$, such that $v_j\to 0$ as $j\to\infty$. 
Choosing an eigenvector $\phi_\tau(v)$ corresponding to $z_\tau(v)$ such that $\phi_\tau(v_j)$ belongs to a Jordan chain of length greater than~1, we get for any $\phi_\tau^*(v_j) \in \clV(z^*_\tau(v_j), N_{v_j}^*)$,
\begin{equation}\label{F: (phi,phi*) = 0}
 \scal{\phi_\tau(v_j), \phi_\tau^*(v_j)} = 0.
\end{equation}
This is because, in general, if an eigenvalue $z_0$ of $N_0$ is not semisimple, with a Jordan chain $\phi_0$, $\phi_1$, then for any $\phi_0^*\in\clV(\bar z_0, N_0^*)$, 
\begin{align*}
0 &= \scal{\phi_0 + (z_0 - N_0)\phi_1,\phi_0^*} \\
  &= \scal{\phi_0,\phi_0^*} + \scal{\phi_1, (\bar z_0 - N_0^*)\phi_0^*} \\
  &= \scal{\phi_0,\phi^*_0}.
\end{align*}
The vector function $\phi_\tau^*(v)$ can be chosen to be continuous, so that sending $j\to\infty$ in~\eqref{F: (phi,phi*) = 0} gives
$
 \scal{\phi_\tau(0), \phi_\tau^*(0)} = 0.
$
Since by~\eqref{E: (phi*,phi) = 0}, $\scal{\phi^*_\mu(0),\phi_\tau(0)} = 0$ for any $\phi^*_\mu(v)$ corresponding to $z^*_\mu(v)$ such that $z_\tau \not\equiv z_\mu$, we get a contradiction.

Proof of the remaining parts of the second statement is omitted (see \cite[Theorem 3.6]{HryLan}).
\end{proof}

\begin{prop}
For Assumption~\ref{A: weak conjugate} to hold for all eigenvalues of the $z_0$-group is equivalent to the statement:
\begin{enumerate}
\item[\upshape{(P)}] The eigenvalue~$z_0$ is semisimple and generating eigenvectors span the eigenspace $\clV(z_0,N_0)$.
\end{enumerate}
\end{prop}

\begin{proof}
Suppose Assumption~\ref{A: weak conjugate} holds for the $z_0$-group. 
Then~$z_0$ must be semisimple by Theorem~\ref{T: assum}. 
Moreover, by the same theorem the spectrum of the~$z_0$-group is given for all small~$v$ by~$m$ (counting multiplicities) analytic eigenvalues $z_\tau(v)$, $j=1,\ldots, m$, for which we can find analytic normalised eigenvectors $\phi_\tau(v)$, $j=1,\ldots, m$. 
It follows that the eigenspace $\clV(z_0,N_0)$ is spanned by the generating eigenvectors~$\phi_\tau(0)$, $j=1,\ldots,m.$

Suppose~(P) holds.
We note that by \cite[Proposition 3.5]{HryLan} there can be no branching at~$z_0$, so all eigenvalues of the $z_0$-group are analytic.
For any eigenvalue $z_\tau(v)$ of the $z_0$-group, there exists a corresponding analytic eigenpath $\phi_\tau(v)$. 
Since~$z_0$ is semisimple, by Proposition~\ref{P: prop X} there must be an eigenvector $\phi^*_0$ corresponding to $\bar z_0$ such that $\scal{\phi^*_0,\phi_\tau(0)} \neq 0$.
It follows from~(P) that generating eigenvectors span the eigenspace $\clV(\bar z_0,N_0^*)$. 
Thus there exists an eigenpath~$\phi^*(v)$ corresponding to an eigenvalue~$z^*(v)$ of~$N_v^*$ such that $\phi^*(0) = \phi^*_0$.
The eigenvalue $z^*(v)$ cannot be different from $z^*_\tau(v)$ by~\eqref{E: (phi*,phi) = 0} (or in this case also Lemma~\ref{L: phi*(mu) and phi(nu) are perp} below), which completes the proof.
\end{proof}

It follows that Assumption~\ref{A: weak conjugate} holds at least when $v=0$ is not an exceptional point and $z_0$ is semisimple. 
It also clearly holds if $z_0$ is a simple eigenvalue.
This assumption is therefore generic in the sense that for random matrices it occurs with probability zero.

\begin{lemma} \label{L: def of phi*(nu)}
Given Assumption~\ref{A: weak conjugate} for the eigenvalues $z_\tau(v)$, $\tau=1,2,\ldots,m$, which split from $z_0\in\sigma_d(N_0)$, 
let~$\phi_{\tau}(v),$ $\tau=1,2\ldots,m$, be corresponding holomorphic eigenvector functions.
Then there exist anti-holomorphic eigenvector functions $\phi^*_{\tau}(v),$ 
taking values in $\clV(z^*_{\tau}(v), N_v^*),$ such that 
\begin{equation} \label{F: (phi* mu,phi nu)=delta}
  \scal{\phi^*_{\mu}(v),\phi_{\tau}(v)} = \delta_{\mu\tau}, 
\end{equation}
which are unique up to rearrangement for repeated eigenvalues.
\end{lemma}
\begin{proof}
The existence of such eigenpaths follows from Theorem~\ref{T: assum}, so it remains to prove their uniqueness. 
Clearly, the vectors $\phi^*_{\mu}(v)$ are linearly independent. 
So, if $\psi^*_{\mu}(v)$ is another set with the same property then for some 
scalar anti-holomorphic functions $\bar c^{\tau}_{\mu}(v)$ we have 
\[
  \psi^*_{\mu}(v) = \phi^*_{\tau}(v) \bar c^{\tau}_{\mu}(v).
\]
Taking the scalar product of both sides with $\phi_{\tau}(v),$ we infer from \eqref{F: (phi* mu,phi nu)=delta} that
$c^{\tau}_{\mu}(v) = \delta_{\mu\tau}.$ 
%(Existence) In general position, this can be deduced from Proposition~\ref{P: phi*(mu) and phi(nu) are perp} and 
%Lemma~\ref{L: phi*(s0) not perp to phi(s0)}. In an exceptional case, perturb~$W$ slightly and take the limit. 
\end{proof}

For a basis of vectors $\phi_{1}(0),\ldots,\phi_{m}(0)$ of $\clV(z_0,N_0),$ we say that the vectors $\phi^*_{1}(0),\ldots,\phi^*_{m}(0)$ from Lemma~\ref{L: def of phi*(nu)} are their \emph{conjugates}.
%We remark that if $\dim \clV(z,N_v) = 1$ then the equality~\eqref{F: (phi*,phi)=1} alone suffices to define the conjugate.
% \[
%   \scal{\phi^*_\mu(s),\phi_\tau(s)} = \delta_{\mu\tau}.
% \]

\medskip
Let's recap the situation when Assumption~\ref{A: weak conjugate} holds for all eigenvalues $z_\tau(v)$, $\tau=1,2,\ldots$ which split from a given eigenvalue~$z_0\in\sigma_d(N_0)$. 
The geometric multiplicity of~$z_0$ is equal to~$m = \dim \Upsilon^1_{z_0}(N_0,W)$ and since it follows from Theorem~\ref{T: assum} that~$z_0$ is semisimple, its algebraic multiplicity is also equal to~$m$. 
The eigenvalue $z_0 = z(0)$ is non-branching and in a neighbourhood of~$0$, it splits into single-valued functions of~$v$ only, namely $z_\tau(v)$. 
On the other hand, the inverse function~$v(z)$ may consist of multi-valued functions with~$z_0$ as a branching point.
Corresponding to $z_\tau(v)$ are analytic eigenpaths~$\phi_\tau(v)$. 
Counting multiplicities, there are~$m$ eigenvalues and~$m$ corresponding eigenpaths.
The vectors $\phi_\tau(0)$, $\tau = 1,\ldots,m$, can be chosen to form a basis of $\clV(z_0,N_0)$. 
Their conjugates $\phi^*_{\tau}(0)$, $\tau = 1,\ldots,m$, form a basis of the eigenspace $\clV(\bar z_0,N_0^*)$.

\begin{remark}
For unbounded perturbations, no changes to this section are required.
\end{remark}

\section{Order of an eigenpath}\label{S: order of e.path}
Let~$\phi(v)$ be an eigenpath of $N_v$ corresponding to an eigenvalue function $z(v).$ 
Following~\cite[Definition 3.1.2]{Aza17}, 
we say that~$\phi(v)$ \emph{has order at least~$k$} at $v=0,$ if the vectors
\begin{equation*} \label{F: order k path}
  W\phi(0), \ W \partial_v \phi(0), \ \ldots, \ W\partial^{k-2}_v \phi(0)
\end{equation*}
are orthogonal to $\clV(\bar z_0,N_0^*).$
If~$\phi^*(v)$ is an eigenpath for $N_v^*$ corresponding to an eigenvalue function~$z^*(v),$ then we also say that $\phi^*(v)$ has order at least~$k$
at~$v=0$ if the vectors 
\begin{equation*} \label{F: order k (*)path}
  W\phi^*(0), \ W\bar \partial_{v}\phi^*(0), \ \ldots, \ W \bar\partial^{k-2}_{ v}\phi^*(0)
\end{equation*}
are orthogonal to $\clV(z_0,N_0).$ 

\begin{lemma} \label{L: chi has order >= k}
Given Assumption~\ref{A: weak conjugate} for an eigenvalue function $z(v)$ of $N_v$, %and such that $z(0) \in \sigma_d(N_0)$ is semisimple.
let $\phi(v)$ be an analytic eigenvector function corresponding to~$z(v)$.
Then the following assertions are equivalent (for $k\geq 1$).
\begin{enumerate}
  \item $\phi(v)$ has order at least~$k$ at~$v=0.$
  \item each vector
   $
     W\phi(0), \ W\partial_v \phi(0), \ \ldots, \ W\partial^{k-2}_v \phi(0)
   $ 
   is orthogonal to a vector $\phi^*\in \clV(\bar z_0,N_0^*),$ 
   such that $\scal{\phi^*,\phi(0)} \neq 0.$ 
  \item 
  $
    \partial_v z(0) = 0, \ \ldots, \ \partial_v^{k-1}z(0) = 0.
  $
  \item for all $j=1,\ldots,k-1,$
  \begin{equation} \label{F: (Hs-z)chi(j)=-jVchi(j-1)}
    (N_0-z_0)\partial_v^j \phi(0) = -j W\partial_v^{j-1}\phi(0).
  \end{equation}
\end{enumerate}
\end{lemma}
\begin{proof}
This proof follows that of~\cite[Lemma 3.1.4]{Aza17} with some adjustments. 
%Nevertheless, for reader's convenience, we give it here. 

First we remark that the existence of a vector $\phi^*\in \clV(\bar z_0,N_0^*)$ not orthogonal to $\phi(0)$ is given by Assumption~\ref{A: weak conjugate}. 
It also follows from Theorem~\ref{T: assum} that~$z(v)$ is holomorphic at~$v=0$, which guarantees the existence of a holomorphic eigenpath~$\phi(v)$.

%Since $\phi^*$ belongs to $\clV(\bar z_0,N_0^*),$
The implication (i) $\Rightarrow$ (ii) is obvious. 
The implication (iv) $\Rightarrow$ (i) is also obvious since 
$\clV(\bar z_0,N_0^*)$ is the kernel of $(N_0 - z_0)^*.$ 

(ii) $\Rightarrow$ (iii): 
$k-1$ times differentiating the eigenvalue equation 
$N_v\phi(v) = z(v)\phi(v)$ gives the equality 
\begin{equation} \label{F: Diff-ed eigenvalue equation(0)}
  (k-1) W \phi^{(k-2)}(v) + N_v \phi^{(k-1)}(v) = \sum_{j=0}^{k-1} \binom{k-1}{j} z^{(j)}(v) \phi^{(k-1-j)}(v).
\end{equation}
Here we let $v = 0$ and take the scalar product of both sides with~$\phi^*.$ 
After cancelling the second summand of the left hand side with the first summand of the right hand side, we obtain the equality
\begin{equation} \label{F: Diff-ed eigenvalue equation}
  (k-1)\scal{\phi^*,W \phi^{(k-2)}(0)} 
     = \sum_{j=1}^{k-1} \binom{k-1}{j} z^{(j)}(0) \scal{\phi^*,\phi^{(k-1-j)}(0)}.
\end{equation}
If $k=2$ then 
\begin{equation*} \label{F: (phi*,W phi)=z'(s)}
  \scal{\phi^*,W \phi(0)} = z'(0) \scal{\phi^*,\phi(0)}.
\end{equation*}
Since by the premise $\scal{\phi^*,\phi(0)} \neq 0,$ 
this equality implies the assertion in the case $k=2.$ 
Assume that the claim holds for $k < n.$ Then from 
\eqref{F: Diff-ed eigenvalue equation} with $k = n,$ using the induction assumption, we get
\[
  (n-1)\scal{\phi^*,W \phi^{(n-2)}(0)} = z^{(n-1)}(0) \scal{\phi^*,\phi(0)}. % = z^{(n-1)}(0).
\]
Since by the premise $\scal{\phi^*,W \phi^{(n-2)}(0)}=0$ and $\scal{\phi^*,\phi(0)}\neq 0,$
this gives~$z^{(n-1)}(0) = 0.$ 

\smallskip
(iii) $\Rightarrow$ (iv): 
We prove the claim for $j=k-1$ but the same proof clearly works for the other values of $j$.
Letting $v=0$ in \eqref{F: Diff-ed eigenvalue equation(0)} gives the equality 
\[
  (k-1) W \phi^{(k-2)}(0) + N_{0} \phi^{(k-1)}(0) 
         = \sum_{j=0}^{k-1} \binom{k-1}{j} z^{(j)}(0) \phi^{(k-1-j)}(0).
\]
By the premise, the right hand side simplifies to~$z(0) \phi^{(k-1)}(0) = z_0 \phi^{(k-1)}(0).$
Hence,
\[
  (N_0-z_0) \phi^{(k-1)}(0) = - (k-1) W \phi^{(k-2)}(0).
\]
\end{proof}

% \begin{rems} \rm In~\cite[Corollary 3.4.7]{Aza16} the vectors~$\phi^{(j)}(v)$
% differ from the ones which we use here by a factor of $j!.$ This explains appearance
% of the coefficient $j$ in the formula \eqref{F: (Hs-z)chi(j)=-jVchi(j-1)}, 
% which is absent in~\cite[Corollary 3.4.7]{Aza16}.
% Another difference is that~$\phi^{(j)}(v)$ in~\cite{Aza16} means a vector of order $j,$
% while in~\cite{Aza17} and here it is a vector of order $j+1$ 
% (to make this notation consistent with standard notation for a higher order derivative).
% \end{rems}

\begin{lemma} \label{L: this is enough} 
Let $k\geq 2,$ and vectors $\phi_0, \phi_1, \ldots, \phi_{k-1}$
be such that for all $j=1,2,\ldots,k-1$
\begin{equation} \label{F: (H-z)phi(j)=-j phi(j-1)}
  (N_0-z_0)\phi_{j} = -j W\phi_{j-1},
\end{equation}
and $\phi_0$ is an eigenvector of $N_0$ corresponding to $z_0.$
Then 
\begin{enumerate}
   \item 
   $
     \phi_0, \ \phi_1, \ \ldots, \ \phi_{k-1}, \ 
   $
   are resonance vectors of orders respectively $1,2,\ldots,k,$
   \item 
   $
     \bfA_{z_0}(N_0,W) \phi_{j} = j \phi_{j-1},
   $
   for any $j=1,2,\ldots,k-1,$
   \item $\phi_0$ is a resonance vector of depth at least~$k-1.$ 
\end{enumerate}
\end{lemma}
\begin{proof}
For~$v$ such that $z_0 \notin \sigma_d(N_v),$ 
the equality \eqref{F: (H-z)phi(j)=-j phi(j-1)} can be rewritten as 
\begin{equation} \label{F: [1+RV]phi(j)=-RV phi(j-1)}
  \big(1 - v R_{z_0}(N_v)W\big) \phi_{j} = -j R_{z_0}(N_v)W \phi_{j-1}.
\end{equation}
The vector $\phi_0$ has order $1$ since it is an eigenvector. 
It follows from the definition~\eqref{F: [1-vRW]k phi=0} of vectors of order $k$ that the operator $R_{z_0}(N_v)W$ preserves the order of vectors, while the operator $(1 - vR_{z_0}(N_v))$ decreases the order of vectors by 1.
Therefore the item~(i) follows by induction from \eqref{F: [1+RV]phi(j)=-RV phi(j-1)}.
According to the definitions \eqref{F: Ak=K(-k)W and Bk=WK(-k)} of $\bfA_{z_0}(N_0,W)$ and \eqref{F: Pz=K0*W, Qz=W*K0} of $P_{z_0}(N_0,W),$ the equality \eqref{F: [1+RV]phi(j)=-RV phi(j-1)} also implies~(ii) after taking contour integrals over a small circle enclosing~$0$ in the~$v$-plane. 
Finally, (ii) plainly implies~(iii) by the definition of depth \eqref{F: depth of res vector}. 
\end{proof}

\begin{lemma} \label{L: phi has order >=k, then ...} 
Given Assumption~\ref{A: weak conjugate} for an eigenvalue function $z(v)$ of $N_v$, with a corresponding analytic eigenvector function $\phi(v)$, 
if the path~$\phi(v)$ has order at least~$k\geq 2$ at~$0,$ then
\begin{enumerate}
   \item  
   $
     \phi(0), \ \phi'(0), \ \ldots, \ \phi^{(k-1)}(0) \ 
   $
   are resonance vectors of orders respectively $1,2,\ldots,k,$
   \item for any $j=1,\ldots,k-1,$
   $
     \bfA_{z_0}(N_0,W) \partial_v^j \phi(0) = j \partial_v^{j-1} \phi(0),
   $
   \item $\phi(0)$ has depth at least~$k-1.$ 
\end{enumerate}
\end{lemma}
\begin{proof}
By Lemma~\ref{L: chi has order >= k}, 
a path~$\phi(v)$ of order $\geq k$ at~$0$ obeys \eqref{F: (Hs-z)chi(j)=-jVchi(j-1)}.
So, the claim follows from Lemma~\ref{L: this is enough}.
\end{proof}

%\medskip

\begin{prop} \label{P: TFAE phi* has order >=k} 
Suppose Assumption~\ref{A: weak conjugate} holds for an eigenvalue function $z(v)$ of $N_v$, and~$\phi(v)$ is a corresponding analytic eigenpath.
Let also $\phi^*(v)$ be an anti-holomorphic eigenpath corresponding to the anti-holomorphic eigenvalue function~$z^*(v)$ of $N_v^*.$ 
Then the following assertions are equivalent.
\begin{enumerate}
  \item $\phi(v)$ has order at least~$k$ at~$v=0.$
  \item
   $
     W\phi^*(0), \ W\bar \partial_{v}\phi^*(0), \ \ldots, \ W\bar\partial^{k-2}_{ v}\phi^*(0)
   $ 
   are orthogonal to a vector $\phi \in \clV(z_0,N_0),$ such that $\scal{\phi^*(0),\phi} \neq 0.$
  \item $\phi^*(v)$ has order at least~$k$ at~$0.$ 
  \item 
  $
    \bar \partial_{v}z^*(0) = 0, \ \ldots, \ \bar\partial^{k-1}_{ v}z^*(0) = 0.
  $
  \item 
  %\begin{equation*} %\label{F: (Hs-z)chi(j)=-jVchi(j-1)}
   $ (N_0^* - \bar z_0)\bar \partial_{v}^j \phi^*(0) = -j W\bar \partial_{v}^{j-1} \phi^*(0), $
  %\end{equation*} 
  for all $j=1,\ldots,k-1.$
\end{enumerate}
\end{prop}
\begin{proof} 
Since $z^*(v) = \overline{z(v)},$ it follows that
$\partial_v^j z(0) = 0$ if and only if $\bar \partial_{v}^j z^*(0) = 0.$
Hence, (iv) is equivalent to Lemma~\ref{L: chi has order >= k}(iii), 
and thus also to (i). 
The equivalence of the other items is proved as in Lemma~\ref{L: chi has order >= k}, 
considering that, by conjugate invariance, 
Assumption~\ref{A: weak conjugate} holds for the eigenvalue $z^*(v)$ of $N_v^*$. 
\end{proof}

It is worth emphasising that an anti-holomorphic eigenpath $\phi^*(v)$
corresponding to~$\phi(v)$ as in this proposition may be non-unique, where by non-uniqueness we of course mean that another such eigenpath may exist which is not a scaling of $\phi^*(v)$. 
Proposition~\ref{P: TFAE phi* has order >=k} holds for all such eigenpaths. 

\begin{lemma} \label{L: last lemma} 
Under the premise of Theorem~\ref{T: d[nu]'s are equal}, let $k\geq 2$ and let $z(v)$ be an eigenvalue function
of $N_v$ with a corresponding eigenvector function $\phi(v).$  
If~$\phi(0)$ has depth at least $k-1,$ then~$\phi(v)$ has order at least~$k$ at~$0.$
\end{lemma}
\begin{proof} 
This proof is an adjustment to this setting of the proofs of 
\cite[Lemma 3.1.6]{Aza17} and the implication $(ii) \Rightarrow (i)$ of~\cite[Theorem 2.3.1]{Aza17}. 
One of the differences is the use of a conjugate vector. %Lemma~\ref{L: phi*(s0) not perp to phi(s0)}. 

We agree to write $\bfA_{z_0}$ for $\bfA_{z_0}(N_0,W)$ and $\bfA_{\bar z_0}$ for $\bfA_{\bar z_0}(N_0^*,W).$
Let $k=2.$ By the definition~\eqref{F: depth of res vector} of depth, there exists a vector $\chi$ such that 
$\bfA_{z_0} \chi = \phi(0).$ So, for any $f \in \clV(\bar z_0,N_0^*),$
we have 
\[
  \scal{f,W\phi(0)} = \scal{f,W\bfA_{z_0} \chi} = 
  \scal{\bfA_{\bar z_0}f,W\chi} = \scal{0,W\chi} = 0,
\]
and so, $\phi(v)$ has order $\geq 2$ at $0.$

Assume the claim for integers less than $k$ and let $\phi(0)$ be of depth $\geq k-1.$ 
Then by the induction assumption the path $\phi(v)$ at~$0$ has order $\geq k-1$. 
Choose an eigenpath $\phi^*(v)$ of $N_v^*,$ corresponding to~$z^*(v),$ so that 
\[
  C := \scal{\phi^*(0),\phi(0)} \neq 0,
\]
which can be done according to Assumption~\ref{A: weak conjugate}. 
By Proposition~\ref{P: TFAE phi* has order >=k}, the eigenpath $\phi^*(v)$ also has order $\geq k-1$.
In particular, 
\[
 \bar\partial_{v}z^*(0) = \ldots = \bar\partial_{v}^{k-1}z^*(0) = 0.
\] 
We combine this with an equation similar to~\eqref{F: Diff-ed eigenvalue equation}, which is obtained by $k-1$ times differentiating the eigenvalue equation $N_v^*\phi^*(v) = z^*(v)\phi^*(v)$ with respect to $\bar v$ at the point $\bar v = 0$ and taking the scalar product with $\phi(0)$.
%\[
% (k-1) W \bar\partial_{v}^{k-2} \phi^*(v) + N_v^* \bar\partial_{v}^{k-1} \phi^*(v) = \sum_{j=0}^{k-1} \binom{k-1}{j} \bar\partial_{v}^{j} z^*(v) \bar\partial_{v}^{k-1-j} \phi^*(v).
%\]
The result is
\[
  (k-1) \scal{\bar\partial_{v}^{k-2} \phi^*(0), W\phi(0)} = \scal{\bar\partial_{v}^{k-1}z^*(0)\phi^*(0),\phi(0)} = Cz^{(k-1)}(0),
\]
where in the last equality we used the fact that $\bar\partial_{v}\bar z$ is the conjugate of $\partial_v z$.  
Since $\phi(0)$ has depth at least $k-1,$ there is some $f$ so that 
$
  \bfA^{k-1}_{z_0} f = \phi(0).
$
Hence, 
\begin{equation*}
 \begin{split}
  C z^{(k-1)}(0) & = (k-1) \scal{\bar\partial_{v}^{k-2} \phi^*(0), W \bfA^{k-1}_{z_0} f} 
  \\ & = (k-1) \scal{\bfA^{k-1}_{\bar z_0} \bar\partial_{v}^{k-2} \phi^*(0), Wf}. 
 \end{split}
\end{equation*}
Since $\phi^*(v)$ at~$0$ has order $\geq k-1,$ the vector $\bar\partial_{v}^{k-2} \phi^*(0)$ has order~$k-1$ by Lemma~\ref{L: phi has order >=k, then ...}(i). 
Therefore
$
  \bfA^{k-1}_{\bar z_0} \bar\partial_{v}^{k-2} \phi^*(0) = 0.
$
Combining this with the previous display gives~$C z^{(k-1)}(0) = 0,$ and since $C$ is non-zero, we have $z^{(k-1)}(0) = 0.$ 
Hence, by Lemma~\ref{L: chi has order >= k}, the path $\phi(v)$ at~$0$ has order $\geq k.$ 
\end{proof}

\begin{cor}\label{C: depths equal}
Suppose Assumption~\ref{A: weak conjugate} holds for an eigenvalue function~$z(v)$ of~$N_v$. 
Let~$\phi(v)$ be a holomorphic eigenpath corresponding to~$z(v)$ and let $\phi^*(v)$ be an anti-holomorphic eigenpath corresponding to the eigenvalue~$z^*(v)$ of~$N^*_v.$ 
The following are equivalent statements.
\begin{enumerate}
\item The order of $\phi(v)$ at $v=0$ is equal to $k$.
\item The order of $\phi^*(v)$ at $v=0$ is equal to $k$.
\item The depth of $\phi(0)$ is equal to $k-1$. 
\item The depth of $\phi^*(0)$ is equal to $k-1$.
\end{enumerate}
\end{cor}
\begin{proof}
By Proposition~\ref{P: TFAE phi* has order >=k} the eigenpaths $\phi(v)$ and $\phi^*(v)$ have the same order.  
By Lemma~\ref{L: phi has order >=k, then ...}(iii), both $\phi_\tau(0)$ and $\phi_\tau^*(0)$ have depth at least $k-1.$
Their depth is at most $k-1$ by Lemma~\ref{L: last lemma}. 
\end{proof}

\medskip
%\section{Different descriptions of $d_\tau$}
Recall that we have
$
  m := \dim \Upsilon^1_{z_0}(N_0,W)
$ 
and 
$
  n := \dim \Upsilon_{z_0}(N_0,W).
$ 
The vector space $\Upsilon_{z_0}(N_0,W)$ is invariant for the nilpotent operator $\bfA_{z_0}(N_0,W).$
The sizes of its Jordan blocks we denote $d_\tau,$ $\tau = 1,\ldots, m.$

A Jordan block is determined by an eigenpath~$\phi_\tau(v)$ corresponding to an eigenvalue~$z_\tau(v)$ which splits from~$z_0$ and satisfies Assumption~\ref{A: weak conjugate}.
Indeed, $\phi_\tau(v)$ determines a set of vectors 
\[
  \frac 1{j!} \phi_\tau^{(j)}(0), \ j=0,1,\ldots,\tilde d_\tau-1,
\]
where $\tilde d_\tau$ is the order of the eigenpath~$\phi_\tau(v)$ at $v=0$.
These vectors form a Jordan chain for $\bfA_{z_0}(N_0,W)$ by Lemma~\ref{L: phi has order >=k, then ...}(ii), which is of maximal length 
by Corollary~\ref{C: depths equal}, hence $\tilde d_\tau = d_\tau$.

Although we assume $z_0$ is non-branching, the inverse function~$v(z)$ may consist of multi-valued functions with~$z_0$ as a branching point. 
As~$z$ makes one round around~$z_0$ these different values of $v(z)$ undergo a permutation.
This permutation gives rise to a partition of the values of $v(z)$ into \emph{cycles}.
There is a correspondence (unique up to multiplicity) between the cycles and the eigenpaths, namely, if $N_v \phi_\tau(v) = z_\tau(v)\phi_\tau(v)$ then the inverse of $z_\tau(v)$ defines the cycle corresponding to~$\phi_\tau(v)$. 
%This can be seen from the lemma below which is adapted from~\cite[Theorem~5.2.1(a)--(c)]{Aza17}.
%\begin{lemma}
%For each cycle $v_\nu^{(\cdot)}(z)$ there exists $\epsilon>0$ and a some element of this cycle, say~$v_\nu^{(0)}(z)$, which takes values on the line $v(z_0) + \mbR$ for all $z\in z_0 + I$, where $I$ is either $[0,\epsilon)$ or $(-\epsilon,0]$.
%The number of such elements in a cycle is either 1 or 2. 
%If there are 2, say~$v_\nu^{(0)}(z)$ and~$v_\nu^{(1)}(z)$, then the numbers $v_\nu^{(0)}(z) - v(z_0)$ and $v_\nu^{(1)}(z) - v(z_0)$ have different signs.
%\end{lemma} 
Let the period of the cycle corresponding to $\phi_\tau(v)$ be denoted $\hat d_\tau$. 
The equivalence of items (i) and (iii) in Lemma~\ref{L: chi has order >= k} shows that 
$
  \hat d_\tau = \tilde d_\tau.
$

Thus, the following analogue of \cite[Theorem 5.2.3]{Aza17} holds.
%Since the proof is quite short, albeit almost verbatim, we give it here.
\begin{thm} \label{T: d[nu]'s are equal}
Suppose Assumption~\ref{A: weak conjugate} holds for all eigenvalue functions~$z_\tau(v)$ of $N_v$ which split from $z_0\in\sigma_d(N_0)$. 
Let~$\phi_\tau(v)$ be corresponding eigenpaths.
Then the following numbers are equal:
\begin{enumerate}
  \item the order $d_\tau$ of the eigenpath $\phi_\tau(v)$ at~$0,$
  \item the period of the $\tau$th cycle of resonance points of the group of~$0,$
  \item \label{item three} the size of the $\tau$th Jordan block for the nilpotent operator $\bfA_{z_0}(N_0,W).$ 
\end{enumerate}
Moreover, the vectors 
\[
  \frac 1{j!} \phi_\tau^{(j)}(0), \ \tau=1,\ldots,m, \ j=0,1,\ldots, d_\tau-1,
\] 
form a Jordan basis for $\bfA_{z_0}(N_0,W).$ 
\end{thm}

%This in particular, completely justifies the equality \eqref{F: depth phi = depth phi*}.
% \begin{proof} By Theorem~\ref{T: TFAE to chi has order k}(iii) 
% the order $d_\tau$ of the eigenpath $\phi_\tau(v)$ at $0$ satisfies 
% \[
%   z_\tau(v) = 0 + \eps_{d_\tau} v^{d_\tau} + O\brs{v^{d_\tau+1}}, \ v \to 0,
% \]
% where $\eps_{d_\tau} \neq 0.$
% This shows (1)$=$(2). Theorem~\ref{T: TFAE to chi has order k}(vi) shows (1)$=$(3)
% and the rest. 
% \end{proof}

We combine Lemmas~\ref{L: chi has order >= k},~\ref{L: phi has order >=k, then ...}
and~\ref{L: last lemma} in the following theorem which includes Theorem~\ref{T: criteria for splitting} as a special case. 
\begin{thm} \label{T: TFAE to chi has order k}
Let Assumption~\ref{A: weak conjugate} be given for an eigenvalue function $z(v)$ of $N_v$ which splits from $z_0\in\sigma_d(N_0)$. 
Let $k\geq 1$ and let~$\phi(v)$ be an eigenvector function corresponding to~$z(v).$ 
Then the following assertions are equivalent:  
\begin{enumerate}
  \item the path~$\phi(v)$ has order~$k$ at~$0.$
  \item the vectors \ 
   $
     W\phi(0), \ W\phi'(0), \ \ldots, \ W\phi^{(k-2)}(0)
   $ 
   \ are orthogonal to some~$\phi^* \in \clV(\bar z_0,N_0^*),$ $\scal{\phi^*,\phi(0)}\neq 0,$
   while $W\phi^{(k-1)}(0)$ is not. 
  \item %\label{TI: TFAE to chi has order k(iii)}
  $
    z'(0) = 0, \ \ldots, \ z^{(k-1)}(0) = 0,
  $
  and $z^{(k)}(0) \neq 0,$
  \item for all $j=1,\ldots,k-1,$ \ \ 
    $  
      (N_0-z_0)\phi^{(j)}(0) = -j W\phi^{(j-1)}(0),
    $  
    and this fails for $j=k.$
  \item %\label{TI: TFAE to chi has order k(v)}
     the vector~$\phi(0)$ has depth~$k-1.$ 
  \item for any $j=1,2,\ldots,k-1,$ \ \ 
   $
     \bfA_{z_0}(N_0,W) \phi^{(j)}(0) = j \phi^{(j-1)}(0),
   $
   and this fails for $j=k.$ 
\end{enumerate}
Moreover, if these conditions hold, then
\begin{enumerate} 
\item[\upshape(vii)] for $j=1,2,\ldots,k,$ \ the vector~$\phi^{(j-1)}(0)$ has order~$j.$
\end{enumerate}
\end{thm}

In~\cite[Lemma 3.1.5]{Aza17} the equality of the item (vi) in this theorem 
was obtained by a different method, which is specific to the case of real~$z_0$ 
(outside essential spectrum) and self-adjoint~$N_0,$ 
and it seems unlikely that this method can be adjusted to the current setting.

We conjecture that item~(vii) is also equivalent to the other items of this theorem. 
For an eigenvalue $z_0 = z(0)$ of geometric multiplicity~$1$ this 
is true, since in this case~(vii) obviously implies~(v).

%\begin{cor} \label{C: criteria for splitting}
%Let~$0$ be a semisimple non-branching point of an eigenvalue function $z(v)$
%of $N_0 + v W$ with a corresponding eigenvector function $\phi(v).$ 
%Then the following assertions are equivalent:  
%\begin{enumerate}
%  \item the path~$\phi(v)$ has order~$\geq 2$ at~$0,$
%  \item $\scal{\phi^*,W\phi(0)} = 0$ 
%     for some~$\phi^* \in \clV(\bar z_0,N_0^*)$ such that $\scal{\phi^*,\phi(0)}\neq 0,$
%  \item $z'(0) = 0,$
%  \item $(N_0-z_0)\phi'(0) = - W\phi(0),$  
%  \item the vector~$\phi(0)$ has depth at least one, 
%  \item $\bfA_{z_0}(N_0,W) \phi'(0) = \phi(0),$
%  \item $z(0)$ is a branching point of the inverse function $v(z).$ 
%\end{enumerate}
%\end{cor}

\begin{cor} $\phi(v)$ cannot have order greater than~$\rank(W).$ 
\end{cor}
\begin{proof} This follows from Theorem~\ref{T: TFAE to chi has order k}(ii), since for $\phi(v)$ of order~$k$ 
at~$0$ the vectors $W\phi(0),$ $\ldots,$ $W\phi^{(k-1)}(0)$ are linearly independent. 
\end{proof}

\begin{remark}
For an unbounded perturbation $W=F^*JF,$ the definition of the order of an eigenpath is interpreted as follows: $\phi(v)$ has order at least~$k$ at $v=0$ if
\begin{equation*} \label{F: unbdd order k path}
  JF\phi(0), \ JF \partial_v \phi(0), \ \ldots, \ JF\partial^{k-2}_v \phi(0)
\end{equation*}
are orthogonal to $F\clV(\bar z_0,N_0^*).$

Items (ii) and (iv) in the statement of Lemma~\ref{L: chi has order >= k} are replaced by the following:
\begin{enumerate}
  \item[(ii$'$)] the vectors
   $
     JF\phi(0), \ JF\partial_v \phi(0), \ \ldots, \ JF\partial^{k-2}_v \phi(0)
   $ 
   are orthogonal to a vector $F\phi^*\in F\clV(\bar z_0,N_0^*),$ 
   such that $\scal{\phi^*,\phi(0)} \neq 0.$ 
  \item[(iv$'$)] for all $j=1,\ldots,k-1,$ and any $f\in\dom[\clA],$
  \begin{equation*}
    (N_0-z_0)[f, \partial_v^j \phi(0)] = -j\scal{Ff, JF\partial_v^{j-1}\phi(0)}.
  \end{equation*}
\end{enumerate}
The changes to the proof are then straightforward.

The statement of Lemma~\ref{L: this is enough} is altered so that~\eqref{F: (H-z)phi(j)=-j phi(j-1)} reads
\begin{equation}\label{F: (N0-z0)[f,phi(j)] = -jW[f,phi(j-1)]}
 (N_0 - z_0)[f, \phi_j] = -j \scal{Ff,JF\phi_{j-1}}, \quad \forall f\in\dom [\clA],
\end{equation}
and within each item of the conclusion $F\phi_j$ appears instead of $\phi_j$.
%\begin{enumerate}
%   \item[(i$'$)] 
%   $
%     F\phi_0, \ F\phi_1, \ \ldots, \ F\phi_{k-1}, \ 
%   $
%   are resonance vectors of orders respectively $1,2,\ldots,k,$
%   \item[(ii$'$)]
%   $
%     \bfA_{z_0}(N_0,W) F\phi_{j} = j F\phi_{j-1},
%   $
%   for any $j=1,2,\ldots,k-1.$
%   \item[(iii$'$)] $F\phi_0$ is a resonance vector of depth at least~$k-1.$ 
%\end{enumerate}

In the proof, instead of~\eqref{F: [1+RV]phi(j)=-RV phi(j-1)}, we obtain
\begin{equation}\label{F: [1 - vTJ]Fphi(j) = -jTJFphi(j-1)}
 \big[1 - v T_{z_0}(N_v)J\big] F\phi_{j} = -j T_{z_0}(N_v)JF \phi_{j-1}
\end{equation}
using a similar technique as in Remark~\ref{R: 1st unbdd pf}: 
The equality~\eqref{F: (N0-z0)[f,phi(j)] = -jW[f,phi(j-1)]} can be rewritten as
\[
(N_v - z_0)[f,\phi_j] - v\scal{Ff,JF\phi_j} = - j \scal{Ff,JF\phi_{j-1}},
\]
which for $f=R_{\bar z_0}(N_0^*)\psi$ and any $\psi\in\clK$ implies the equality of vectors
\[
 \phi_j - v(FR_{\bar z_0}(N_0^*))^*JF\phi_j = - j (FR_{\bar z_0}(N_0^*))^*JF\phi_{j-1}.
\]
Equality~\eqref{F: [1 - vTJ]Fphi(j) = -jTJFphi(j-1)} follows after applying~$F.$
Given~\eqref{F: [1 - vTJ]Fphi(j) = -jTJFphi(j-1)}, the rest of the argument is virtually unchanged.

The necessary changes to the rest of the results in this document are minimal and should be clear given those previous and so this will be the last remark concerning the case of unbounded perturbations.
\end{remark}

\section{Structure of the resonance projection}\label{S: structure of Pz}
\subsection{Decomposition of \texorpdfstring{$P_z$}{Pz}}\label{S: Pz decomp} Here we show that the resonance projection $P_{z_0}(N_0,W)$ can be decomposed into a sum of projections corresponding to cycles of resonance points.

\begin{lemma} \label{L: phi*(mu) and phi(nu) are perp} 
Assume that $z_\tau(v)$ and $z_\mu(v)$ are two different holomorphic eigenvalue functions of $N_v$,
and $\phi_\tau(v)$ and $\phi^*_\mu(v)$ are eigenvector functions corresponding to $z_\tau(v)$ 
and $z^*_\mu(v)$ respectively.
Then 
%\begin{equation*} \label{F: (phi* mu,phi nu) = 0}
$
  \scal{\phi^*_\mu(v),\phi_\tau(v)} = 0.
$
%\end{equation*}
\end{lemma}

\begin{proof} This proof uses a standard argument used to prove orthogonality of eigenvectors
of self-adjoint operators. 
Using the eigenvalue equations for $\phi_\tau(v)$ and $\phi_\mu^*(v)$, %\eqref{F: N(v)phi(v)=z(v)phi(v)} and \eqref{F: H(bar s) chi(bar s)=bar z chi(bar s)}, 
we have 
\begin{equation*}
  \begin{split}
    z_\mu(v) \scal{\phi^*_\mu(v),\phi_\tau(v)} & = \scal{z^*_\mu(v)\phi^*_\mu(v),\phi_\tau(v)}
         = \scal{N_v^* \phi^*_\mu(v),\phi_\tau(v)}
    \\ & = \scal{\phi^*_\mu(v),N_{v} \phi_\tau(v)}
         = \scal{\phi^*_\mu(v),z_\tau(v)\phi_\tau(v)}
    \\ & = z_\tau(v)\scal{\phi^*_\mu(v),\phi_\tau(v)}.
  \end{split}
\end{equation*}
Since the functions $z_\tau(v), z_\mu(v)$ and $\scal{\phi^*_\mu(v),\phi_\tau(v)}$ are holomorphic
and $z_\tau(v) \neq z_\mu(v),$ the scalar product must vanish.
\end{proof}

%\begin{lemma}\label{L: 3.1.10}
%In the setting of Lemma~\ref{L: def of phi*(nu)}, if $\mu\neq\tau$ then 
%$
%  \scal{\phi^*_\tau(0),W\phi_\mu(0)} = 0.
%$
%\end{lemma}
%\begin{proof}
%We use the same argument as in~\cite[Lemma 3.1.10]{Aza17}. 
%From \eqref{F: (phi* mu,phi nu)=delta},
%\[
% 0 = \scal{\phi^*_\tau(v),z_\mu(v)\phi_\mu(v)} = \scal{\phi^*_\tau(v),N_v\phi_\mu(v)},
%\]
%which when differentiated at $v=0$ gives 
%\begin{align*}
% 0 &= \scal{\bar \partial_{v}\phi^*_\tau(0),N_0\phi_\mu(0)} + \scal{\phi^*_\tau(0),W\phi_\mu(0)} + \scal{\phi^*_\tau(0),N_0\partial_{v}\phi_\mu(0)}
% \\ &=  \scal{\phi^*_\tau(0),W\phi_\mu(0)} + z_0\left(\scal{\bar \partial_{v}\phi^*_\tau(0),\phi_\mu(0)} + \scal{\phi^*_\tau(0),\partial_{v}\phi_\mu(0)}\right)
% \\&= \scal{\phi^*_\tau(0),W\phi_\mu(0)},
%\end{align*}
%where the last equality follows by differentiating~\eqref{F: (phi* mu,phi nu)=delta}.
%\end{proof}

%Proof of the following theorem is the same as that of~\cite[Theorem 3.1.12]{Aza17} with some obvious and mostly notational adjustments, but we present it here for convenience.
\begin{thm}\label{T: (phi*^(j),W phi^(k)) = 0}
Let~Assumption~\ref{A: weak conjugate} hold for two distinct eigenvalue functions $z_\mu(v)$ and $z_\tau(v)$, with corresponding eigenvector functions $\phi_\mu(v)$ and $\phi_\tau(v)$.
If the functions~$\phi_\mu(v)$ and $\phi_\tau(v)$ have 
orders $d_\mu$ and $d_\tau$ respectively at~$0$,
then for any $j=0,1,2,\ldots,d_\mu-1$ and $k=0,1,2,\ldots,d_\tau-1$,
\begin{equation}\label{F: (phi*^(j),W phi^(k)) = 0}
  \scal{\bar \partial_{v}^j \phi^*_\mu({0}),W\partial_v^k \phi_\tau({0})} = 0.
\end{equation}
%where for an eigenpath $\phi(s),$ \ $\bar \partial_{s}^j \phi^*(s) = {d^j\phi^*(s)} / {d\bar s^j}.$
\end{thm}

\begin{proof}(cf.~\cite[Theorem 3.1.12]{Aza17})
%It is clearly sufficient to consider the case $j=d_\mu-1$, $k=d_nu-1$.
From Lemma~\ref{L: phi*(mu) and phi(nu) are perp}, we have
\[
 \scal{\phi^*_\mu(v),N_v\phi_\tau(v)} = z_\tau(v)\scal{\phi^*_\mu(v),\phi_\tau(v)} = 0,
\]
which when differentiated $m + 1 := j + k + 1$ times at $v=0$ using the Leibniz rule gives
\begin{equation}\label{F: Leibniz rule}
 0 = \sum_{l=0}^{m+1}\binom{m+1}{l}\scal{\bar\partial_v^{l}\phi^*_\mu(0),N_0\partial_v^{m+1-l}\phi_\tau(0)} + \sum_{l=0}^{m}(m+1)\binom{m}{l}\scal{\bar\partial_v^{l}\phi^*_\mu(0),W\partial_v^{m-l}\phi_\tau(0)}.
\end{equation}
We transform the first summand as 
\begin{multline*}
(E):= \sum_{l=0}^{m+1}\binom{m+1}{l}\scal{\bar\partial_v^{l}\phi^*_\mu(0),N_0\partial_v^{m+1-l}\phi_\tau(0)} 
\\ = \sum_{l=0}^{j}\binom{m+1}{l}\scal{N_0^*\bar\partial_v^{l}\phi^*_\mu(0),\partial_v^{m+1-l}\phi_\tau(0)} + \sum_{l=j+1}^{m+1}\binom{m+1}{l}\scal{\bar\partial_v^{l}\phi^*_\mu(0),N_0\partial_v^{m+1-l}\phi_\tau(0)}
\end{multline*}
and apply the equalities 
\begin{gather*}
 N_0^*\bar\partial_v^{l}\phi^*_\mu(0) = \bar z_0\bar\partial_v^{l}\phi^*_\mu(0) - lW\bar\partial_v^{l-1}\phi^*_\mu(0), \\  N_0\partial_v^{m+1-l}\phi_\tau(0) = z_0\partial_v^{m+1-l}\phi_\tau(0) - (m+1-l)W\partial_v^{m-l}\phi_\tau(0),
\end{gather*}
from Lemma~\ref{L: chi has order >= k} and Proposition~\ref{P: TFAE phi* has order >=k} to obtain
\begin{align*}
(E) &= z_0\sum_{l=0}^{m+1}\binom{m+1}{l}\scal{\bar\partial_v^{l}\phi^*_\mu(0),\partial_v^{m+1-l}\phi_\tau(0)} - \sum_{l=1}^{j}l\binom{m+1}{l}\scal{W\bar\partial_v^{l-1}\phi^*_\mu(0),\partial_v^{m+1-l}\phi_\tau(0)} 
\\&\qquad -  \sum_{l=j+1}^{m}(m+1-l)\binom{m+1}{l}\scal{\bar\partial_v^{l}\phi^*_\mu(0),W\partial_v^{m-l}\phi_\tau(0)}
\\& = z_0\partial_v^{m+1}\scal{\phi^*_\mu(v),\phi_\tau(v)}\Big|_{v=0} - \sum_{l=0}^{j-1}\frac{(m+1)!}{l!(m-l)!}\scal{W\bar\partial_v^{l}\phi^*_\mu(0),\partial_v^{m-l}\phi_\tau(0)} 
\\&\qquad -  \sum_{l=j+1}^{m}\frac{(m+1)!}{l!(m-l)!}\scal{\bar\partial_v^{l}\phi^*_\mu(0),W\partial_v^{m-l}\phi_\tau(0)}.
\end{align*}
The first summand here is zero by Lemma~\ref{L: phi*(mu) and phi(nu) are perp} and hence returning to~\eqref{F: Leibniz rule}, we find~\eqref{F: (phi*^(j),W phi^(k)) = 0}.
\end{proof}

%The next theorem is an analogue of \cite[Theorem 5.3.1]{Aza17} and again the proof is virtually unchanged.
In the next theorem we consider the total projection corresponding to a cycle of resonance points.

\begin{thm} \label{T: properties of P[nu]}
Let Assumption~\ref{A: weak conjugate} be given for all eigenvalue functions~$z_\tau(v)$ of~$N_v$ which split from $z_0\in\sigma_d(N_0)$ and let~$\phi_\tau(v)$ be corresponding eigenvector functions for $\tau=1,\ldots,m.$ 
For each $\tau = 1,\ldots, m,$ let $v_\tau^{(j)}(z),$ $j=0,\ldots,d_\tau-1,$ be the corresponding cycle of resonance points which constitute the multivalued inverse of~$z_\tau(v).$ 
Then the following assertions hold:\\
\begin{enumerate}
  \item \label{P: P[nu](z) is holomorphic} The function 
\[
 P_{z}^{[\tau]} := \sum_{j=0}^{d_\tau-1} P_{z_\tau(v)}\left(N_{v_\tau^{(j)}(z)},W\right)
\]
of~$z$ is holomorphic in a neighbourhood of~$z_0.$
  \item \label{C: dim Upsilon[nu]=d(nu)}
The range  \label{Page: Upsilon lambda nu}
\[
  \Upsilon_{z_0}^{[\tau]} := \im P_{z_0}^{[\tau]}
\]
of the idempotent $P_{z_0}^{[\tau]}$ (which exists by item \ref{P: P[nu](z) is holomorphic}) 
has dimension $d_\tau$ and the vectors 
\begin{equation*} \label{F: phi(j) is a basis}
  \phi_\tau(0), \ \phi'_\tau(0), \ \ldots, \ \phi^{(d_\tau-1)}_\tau(0).
\end{equation*}
form its basis. 
  \item \label{C: A(lamb)(nu) is reduced}
  The operators $P_{z_0}^{[\tau]}(N_0,W)$ and~$\bfA_{z_0}(N_0,W)$ commute
  and thus the vector space $\Upsilon_{z_0}^{[\tau]}$ is invariant with respect to $\bfA_{z_0}(N_0,W).$ 
  \item \label{L: dim(Ups[nu]cap clV)=1}
   The dimension of the vector space $\Upsilon_{z_0}^{[\tau]} \cap \clV(z_0,N_0)$ is equal to $1.$
  \item \label{L: bfA[nu] is cyclic} 
  The Jordan cell decomposition of 
the restriction of the operator $\bfA_{z_0}^{[\tau]}$ to $\Upsilon_{z_0}^{[\tau]}$
consists of only one Jordan cell. 
\end{enumerate}
\end{thm}
\begin{proof}(cf.~\cite[Theorem 5.3.1]{Aza17})
Since all of the resonance points $v_\tau^{(j)}(z),$ $j=0,\ldots,d_\tau-1,$ converge to~$0$ as $z\to z_0,$
it follows that the $j$th divided difference $\phi_\tau^{[j]}$ evaluated at the points $v_\tau^{(0)}(z),\ldots,v_\tau^{(j)}(z)$ converges to $\frac{1}{j!}\phi_\tau^{(j)}(0).$
Then because 
\[
 P_z^{[\tau]}\phi_\tau^{[j]}\big(v_\tau^{(0)}(z),\ldots,v_\tau^{(j)}(z)\big) = \phi_\tau^{[j]}\big(v_\tau^{(0)}(z),\ldots,v_\tau^{(j)}(z)\big),
\]
by taking the limit as $z\to z_0$ it can be concluded that the $d_\tau$-dimensional range of $P_{z}^{[\tau]}$ converges to the linear span of the independent vectors $\phi^{(j)}(0),$ $j=0,\ldots,d_\tau-1.$ 

Each $P_z^{[\tau]},$ $\tau=1,\ldots,m,$ is by construction single-valued and its Laurent expansion can only have finitely many terms with negative powers of $(z-z_0)$ (see e.g.~\cite[Theorem~II-1.8]{Kat80}). 
Since the total projection $\sum_{\tau=1}^m P^{[\tau]}_z$ converges to $P_{z_0}(N_0,W)$ and is therefore bounded as $z\to z_0,$ any terms with negative powers of $P_z^{[\tau]}$, $\tau=1,\ldots,m,$ must cancel one another as $z\to z_0.$ 
However, they cannot interact in the sense that $P_z^{[\tau]}P_z^{[\mu]} = \delta_{\tau\mu}P_z^{[\tau]}.$
Hence it must be that each $P_z^{[\tau]}$ is also bounded in a neighbourhood of~$z_0.$
This proves (i).

We have 
\[
 P_{z_0}^{[\tau]}P_{z_0}^{[\mu]} = \delta_{\tau\mu}P_{z_0}^{[\tau]}
\]
and
\[
 P_{z_0}(N_0,W) = \sum_{\tau=1}^m P_{z_0}^{[\tau]}.
\]
Using these equalities and the Laurent expansion of $R_{z_0}(N_v),$ (iii) is proved by comparing coefficients on both sides of the equality $P_{z_0}^{[\tau]}R_{z_0}(N_v)W = R_{z_0}(N_v)WP_{z_0}^{[\tau]}.$

It follows from (iii) that the dimension of $\Upsilon_{z_0}^{[\tau]}\cap\clV_{z_0}$ is at least 1, but it cannot be more since there is a correspondence of cycles and eigenvectors, which proves (iv).

Finally, (v) follows directly from (iii) and (iv).
\end{proof}
%In fact, all results of \cite[Subsections 5.3--5.6]{Aza17} hold almost verbatim. 

\subsection{Schmidt representation for \texorpdfstring{$P_z$}{Pz}}\label{S: Pz rep}

%An analogue of \cite[Theorem 5.9.1]{Aza17} holds:
\begin{thm} \label{T: P = sum sum} 
Let Assumption~\ref{A: weak conjugate} hold for all eigenvalue functions~$z_\tau(v)$ of~$N_v$ which split from the eigenvalue~$z_0$ of~$N_0$ of geometric multiplicity~$m.$
Let $\phi_\tau(v)$, $\tau=1,\ldots,m$, be corresponding eigenpaths.
Then the idempotent operator~$P_{z_0}(N_0,W)$ can be written in the form 
\begin{equation*} %\label{F: P = sum sum ...}
  P_{z_0}(N_0,W) = \sum_{\tau=1}^m \sum_{k,j=0}^{d_\tau-1} 
      \frac 1{k!j!} \alpha_{\tau}^{kj} \scal{W \bar \partial_{v}^{k} \phi^*_\tau(0),\cdot} \partial_v^j\phi_\tau(0).
\end{equation*}
where $\alpha_\tau$ is a skew-upper-triangular Hankel matrix of size $d_\tau \times d_\tau.$
\end{thm}
This theorem is an analogue of \cite[Theorem 5.9.1]{Aza17}, where in this case the matrix $\alpha$ is not necessarily real. 

\begin{proof}
%Recall that a matrix $A=(a_{ij})$ is called a {\em Hankel matrix} if its entries $a_{ij}$ depend only on $i+j,$ that is, all skew-diagonals of $A$ are constant.
By Theorem~\ref{T: properties of P[nu]}, a natural Jordan basis for the vector space $\im P_{z_0}(N_0,W)$ is given by
\[
 \frac{1}{j!}\partial_v^j\phi_\tau(0), \qquad \tau = 1,\ldots,m, \ j = 0,1,\ldots,d_\tau-1.
\]
It also follows from previous material that %\eqref{F: Pz=K0*W, Qz=W*K0} and~\eqref{F: P*=Q and A*=B}...  
\[
 \frac{1}{k!}W\bar\partial_v^k\phi^*_\mu(0), \qquad \mu = 1,\ldots,m, \ k = 0,1,\ldots,d_\mu-1,
\]
is a basis of $\im (P_{z_0}(N_0,W))^* = W\im P_{\bar z_0}(N_0^*,W).$
Therefore, there must exist a unique invertible $n\times n$ matrix $\alpha = (\alpha_{\mu\tau}^{kj})$ such that
\[
 P_{z_0}(N_0,W) = \sum_{\mu=1}^m\sum_{k=0}^{d_\mu-1}\sum_{\tau=1}^m\sum_{j=0}^{d_\tau-1} 
      \frac 1{k!j!} \alpha_{\mu\tau}^{kj} \scal{W \bar \partial_{v}^{k} \phi^*_\mu(0),\cdot} \partial_v^j\phi_\tau(0).
\]
It remains to see that $\alpha$ is a direct sum of skew-upper-triangular Hankel matrices.

Since $P_{z_0}$ leaves the vectors $\partial_v^j\phi_\tau(0)$ invariant, we have
\[
 \frac{1}{j'!}\partial_v^{j'}\phi_{\tau'}(0) = \frac{1}{j'!}P_{z_0}(N_0,W)\partial_v^{j'}\phi_{\tau'}(0) = \sum_{\tau,j} \left(\sum_{\mu,k} \alpha_{\mu\tau}^{kj}\beta_{\mu\tau'}^{kj'} \right)\frac 1{j!}\partial_v^{j}\phi_{\tau}(0),
\]
where the matrix $\beta$ is given by
\[
 \beta_{\mu\tau'}^{kj'} = \frac{1}{k!j'!}\scal{\partial_v^{k}\phi^*_{\mu}(0), W\partial_v^{j'}\phi_{\tau'}(0)}.
\]
Hence it must be that $\sum_{\mu,k} \alpha_{\mu\tau}^{kj}\beta_{\mu\tau'}^{kj'} = \delta_{\tau\tau'}\delta_{jj'}$, that is, $\alpha$ is the inverse transpose of $\beta.$ 

Theorem~\ref{T: (phi*^(j),W phi^(k)) = 0} implies $\beta_{\mu\tau}^{kj} = \delta_{\mu\tau}\beta_{\mu\tau}^{kj},$ or in other words, $\beta$ is a direct sum of $m$ matirces of size $d_\tau\times d_\tau.$
Moreover, using Theorem~\ref{T: TFAE to chi has order k}(vi),
\begin{align*}
\beta_{\tau\tau}^{kj} &= \frac{1}{k!j!}\scal{\partial_v^{k}\phi^*_{\tau}(0), W\partial_v^{j}\phi_{\tau}(0)}
 \\ &= \frac{1}{k!(j+1)!}\scal{\partial_v^{k}\phi^*_{\tau}(0), W\bfA_{z_0}(N_0,W)\partial_v^{j+1}\phi_{\tau}(0)}
 \\ &= \frac{1}{k!(j+1)!}\scal{\bfA_{\bar z_0}(N^*_0,W)\partial_v^{k}\phi^*_{\tau}(0), W\partial_v^{j+1}\phi_{\tau}(0)}
 \\ &= \frac{1}{(k-1)!(j+1)!}\scal{\partial_v^{k-1}\phi^*_{\tau}(0), W\partial_v^{j+1}\phi_{\tau}(0)}
 \\ &= \beta_{\tau\tau}^{k-1,j-1}.
\end{align*}
By the same argument, if $k+j+1 \leq d_\tau-1$ then
\begin{align*}
 \beta_{\tau\tau}^{kj} = \beta_{\tau\tau}^{0,j+k} &= \frac{1}{(j+k)!}\scal{\phi^*_{\tau}(0), W\partial_v^{j+k}\phi_{\tau}(0)} \\ &= \frac{1}{(j+k)!}\scal{\bfA_{\bar z_0}(N^*_0,W)\phi^*_{\tau}(0), W\partial_v^{j+k+1}\phi_{\tau}(0)} = 0.
\end{align*}
Thus the matrix $\beta$ is a direct sum of Hankel matrices with zeros above the main skew-diagonal.
It follows that $\alpha = \beta^{-1}$ is a direct sum of Hankel matrices with zeros below the main skew-diagonal.
\end{proof}

%The difference between this theorem and~\cite[Theorem 5.9.1]{Aza17} is that the matrix $\alpha$ is not necessarily real. 

\begin{comment}
\begin{prop} 
Let~$0$ be a semisimple non-branching point for the eigenvalue function~$z(v)$ of $N_v$ of geometric multiplicity~$m.$
Then~$K_0(z_0,N_0,W)$ can be written in the form 
\begin{equation*} %\label{F: P = sum sum ...}
  K_0(z_0,N_0,W) = \sum_{\tau=1}^m \sum_{k,j=0}^{d_\tau-1} 
      \frac 1{k!j!} \alpha_{\tau}^{kj} \scal{\bar \partial_{v}^{k} \phi^*_\tau(0),\cdot} \partial_v^j\phi_\tau(0).
\end{equation*}
where $\alpha_\tau$ is a skew-upper-triangular 
Hankel matrix of size $d_\tau \times d_\tau.$
\end{prop}
\begin{proof} Let $L_0$ be the right hand side. From Theorem~\ref{T: P = sum sum}, 
we have $L_0W = P_{z_0}$ and, by \eqref{F: Ak=K(-k)W and Bk=WK(-k)}, $K_0W = P_{z_0}.$ So 
\begin{equation}\label{F: L0W=K0W}
  L_0W = K_0W.
\end{equation}
From the definition of $L_0$ we have $\im(L_0) \subset \im (P_{z_0});$
combining this with~\eqref{F: L0W=K0W} gives 
\[\im(L_0) = \im(P_{z_0}).\]
Both $K_0$ and $L_0$ are zero on $\im(W)^\perp$ and $W$ is invertible on $\im(W)$,
so $L_0=K_0.$ 
%(DETAILS).
\end{proof}
A similar formula for $K_{-j}$ for $j=1,2,\ldots$ can be obtained from the formula
for $\bfA^j_{z_0}.$ 
\end{comment}

%\chapter{Tangent and high order directions}

\section{Order and tangency to the resonance set}\label{S: tangent high order}
For a given eigenvalue $z_0\in\sigma_d(N_0)$, we say a direction $W$ has order~$k$ if the triple $(z_0,N_0,W)$ has order~$k$. %as defined in Section~\ref{S: prelim}.
In this section we demonstrate a direct connection between high order directions~$W$ and those that are tangent to the complex variety of operators in~$\clA$ which have $z_0$ as an eigenvalue. 
This connection was observed in~\cite{Aza17} in the case of a real isolated eigenvalue~$z_0$ and self-adjoint~$N_0$. % and the real variety of relatively compact self-adjoint perturbations with~$z_0$ as an eigenvalue. 
%Given all the preparatory work done in the previous sections, an extension of the results of~\cite{Aza17} of this type to the off-axis setting turns out to be more or less straightforward. 
%The exposition is mostly self-contained, though many ideas come from~\cite{Aza17}.

Given a fixed number~$z_0$ outside the essential spectrum of the affine space $\clA$, %(defined in Section~\ref{S: prelim}),
the \emph{resonance set} $\euR(z_0)$ is defined by 
\[
  \euR(z_0) := \{N \in \clA \colon z_0 \in \sigma_d(N)\}.
\]
We say that an analytic path of operators $N(v)$ in $\clA$ is \emph{resonant} (at~$z_0$),
if $N(v) \in \euR(z_0)$ for all~$v.$  
If an analytic path is not resonant, then the following trivial lemma holds. 
\begin{lemma} For an analytic curve $\gamma$ in $\clA$ which is not contained in $\euR(z_0),$ 
the set $\euR(z_0) \cap \gamma$ is discrete. 
\end{lemma}

%Definitions of a tangent direction~$W$ etc.~hold verbatim in this case. 
Let~$z_0 \in \sigma_d(N_0),$ where $N_0 \in \clA$, and let $k\in\mbN$. 
We say that a direction~$W \in \clA_0$ is \emph{tangent to order} at least $k$ at $N_0,$
if there exists a resonant path $N(v)$ such that for some numbers $c_2, c_3, \ldots, c_{k-1},$ 
\begin{equation}\label{F: resonant path}
  N(v) = N_0 + v W + \sum_{j=2}^{k-1} c_j v^j W + O(v^k), \ v \to 0.
\end{equation}
We say that~$W$ is \emph{tangent} at $N_0$ if it is tangent to order at least two.
Otherwise, we say that~$W$ is \emph{transversal} at $N_0.$
By changing the parameter $v$, the coefficients $c_j$, $j=2,\ldots,k-1$, can be eliminated so that
\[
 N(v) = N_0 + v W + O(v^k), \ v \to 0.
\]
A resonant path of this form is called {\em standard}.

% \begin{lemma} Let $z_0 \in \sigma_d(N_0).$ If a direction~$W$ is tangent to $\euR(z_0)$ at $N_{0}$ 
% then~$W$ has order at least two.
% \end{lemma}
% \begin{proof} For a resonance path $N(v)$ with $N(0) = N_0$ and $N'(0) = W,$ let
% $\chi(v)$ be a corresponding eigenpath. From the eigenvalue equation 
% \begin{equation} \label{F: N(v)chi(v)=z0 chi(v)}
%   N(v) \chi(v) = z_0 \chi(v),
% \end{equation}
% we get
% $
%   N'(v) \chi(v) + N(v) \chi'(v) = z_0 \chi'(v).
% $
% Letting here $v = 0$ gives $(N_0 - z_0) \chi'(0) = - W \chi(0).$ 
% By Lemma~\ref{L: this is enough} this equality implies that $\chi'(0)$ is a vector of order two, 
% and therefore~$W$ has order at least two. 
% \end{proof}

\subsection{Tangent directions have high orders}
The next theorem is an analogue of~\cite[Theorem 4.1.2]{Aza17}.

\begin{thm}\label{T: tangent high order}
Let $k\geq 2,$ let $z_0$ be a semisimple eigenvalue of $N_0$ and let~$W$ be a direction at~$N_0$ for which Assumption~\ref{A: weak conjugate} holds for all eigenvalues of the~$z_0$-group.
If $N(v)$ is a resonant path tangent to~$W$ at~$N_0$ to order at least~$k$ 
and if $\chi(v)$ is a corresponding analytic eigenpath, then 
\begin{enumerate} 
  \item the vectors $\chi(0), \ \chi'(0), \ \ldots,\ \chi^{(k-1)}(0)$ have orders respectively
  $1,2, \ldots, k,$
  \item the direction~$W$ has order at least $k,$
  \item for any $l=1,2,\ldots,k$
  \[
    \bfA_{z_0}(N_0,W) \chi^{(l-1)}(0) = (l-1) \chi^{(l-2)}(0) + 
        \sum_{j=2}^{l-1} j!\binom{l-1}{j} c_j \chi^{(l-1-j)}(0), 
  \]
  where the numbers $c_j$ are as in~\eqref{F: resonant path}, and 
  \item the eigenvector $\chi(0)$ has depth at least $k-1.$ 
\end{enumerate}
\end{thm}
\begin{proof}(cf.~\cite[Theorem 4.1.2]{Aza17}) 
%This proof follows that of~\cite[Theorem 4.1.2]{Aza17} with some minor obvious modifications. 
We~prove the first item by induction.
The vector $\chi(0)$ has order~$1$ since it is an eigenvector.
Now assume the assertion holds for all numbers up to but not including~$l$.
Since the path~$N(v)$, which is tangent to~$W$ at~$N_0$ to order at least~$l$, is also tangent to order at least $l-1$, it follows from this assumption that the vectors $\chi(0), \chi'(0), \ldots, \chi^{(l-2)}(0)$ have orders $1,2,\ldots,l-1$ respectively.
Differentiating the eigenvalue equation $N(v)\chi(v)=z_0\chi(v)$ $l-1$~times at $v=0$ gives
\begin{align*}
 z_0\chi^{(l-1)}(0) &= \sum_{j=0}^{l-1}\binom{l-1}{j}N^{(j)}(0)\chi^{(l-1-j)}(0)
 \\ &= N_0\chi^{(l-1)}(0) + (l-1)W\chi^{(l-2)}(0) + \sum_{j=2}^{l-1}j!\binom{l-1}{j}c_jW\chi^{(l-1-j)}(0),
\end{align*}
where the second line follows from the tangency of $N(v)$ to $W$ at $N_0$ to order~$l$.
Choosing~$v$ so that~$z_0$ belongs to the resolvent set of $N_v = N_0 + vW,$ the above equality can be rewritten as
\[
 (N_v - z_0)\chi^{(l-1)}(0) -vW\chi^{(l-1)}(0) = -(l-1)W\chi^{(l-2)}(0)- \sum_{j=2}^{l-1}j!\binom{l-1}{j}c_jW\chi^{(l-1-j)}(0),
\]
which we multiply by $R_{z_0}(N_v)$ in order to obtain
\begin{multline}\label{F: (1-vRW)chi^(l-1) = sum}
 (1 - vR_{z_0}(N_v)W)\chi^{(l-1)}(0) \\ = - R_{z_0}(N_v)W\Bigg((l-1)\chi^{(l-2)}(0) + \sum_{j=2}^{l-1}j!\binom{l-1}{j}c_j \chi^{(l-1-j)}(0)\Bigg).
\end{multline}
%It follows from the definition~\eqref{F: [1-vRW]k phi=0} of vectors of order $k$ that the operator $R_{z_0}(N_v)W$ preserves the order of vectors, while the operator $(1 - vR_{z_0}(N_v))$ decreases the order of vectors by 1.
Recall that $R_{z_0}(N_v)W$ preserves the order of resonance vectors whereas $(1 - v R_{z_0}(N_v)W)$ decreases it by 1.
Therefore, since the vector on the right of~\eqref{F: (1-vRW)chi^(l-1) = sum} has order $l-1,$ the vector $\chi^{(l-1)}(0)$ must have order~$l$.
This completes the proof of~(i).

(ii) follows immediately from (i).

To prove (iii), we integrate~\eqref{F: (1-vRW)chi^(l-1) = sum} with respect to~$v$ along a contour enclosing~$0$ and refer to the definitions~\eqref{F: Kj = oint R vj},~\eqref{F: Pz=K0*W, Qz=W*K0}, and~\eqref{F: Ak=K(-k)W and Bk=WK(-k)} of $P_{z_0}(N_0,W)$ and $\bfA_{z_0}(N_0,W).$

(iv) follows immediately from (iii).
\end{proof}

An immediate corollary is the following analogue of~\cite[Proposition 4.1.4]{Aza17}.
\begin{cor}
Under the assumptions of Theorem~\ref{T: tangent high order}, the path $N(v)$ is standard iff
\[
\bfA_{z_0}(N_0,W)\chi^{(l-1)}(0) = (l-1)\chi^{(l-2)}(0), \quad l=1,\ldots,k.
\]
\end{cor}

\subsection{High order directions are tangent}\label{S: high order tangent}
% Unlike proof of Theorem 4.1.2 from~\cite{Aza17}, proof of Theorem 4.3.1 from~\cite{Aza17}
% is not directly applicable in the present setting. 
% Firstly, we need an analogue of~\cite[Theorem 4.2.1]{Aza17}.
% Given this theorem,~\cite[Theorem 4.2.3]{Aza17} and its proof seem to be straightforward. 
% But proof of~\cite[Theorem 4.3.1]{Aza17} uses~\cite[Theorem 2.3.1]{Aza17} which we don't have.

The following is an off-real-axis analogue of~\cite[Theorem 4.2.1]{Aza17}.
\begin{thm} \label{T: intersection is a curve}
Let~$z_0$ be a simple eigenvalue of $N_0$, 
$\chi$ a corresponding normalised eigenvector of~$N_0$ and~$W_0 = \scal{\chi,\cdot}\chi.$
Then the intersection of the affine space 
\[
  \alpha := N_0 + \mbC W + \mbC W_0
\]
with a sufficiently small neighbourhood of the point~$N_0$ in $\Rset$
consists of one and only one simple complex analytic curve~$\gamma.$ 
\end{thm}
\begin{proof} First we show that any deleted neighbourhood of $N_0$ in the affine plane $\alpha$ 
has a resonance point. 
Assume the contrary: there exists a convex neighbourhood~$O$ of $N_0$ in 
$\alpha$ which does not contain resonance points other than $N_0$. 
Choose positive~$s_0,\eps_0$ so that 
\[
  N_0 + \eps_0 \bar \mbD W + s_0 \bar \mbD W_0 \subset O,
\]
where $\bar \mbD$ denotes the closed unit disk.
Soon we will also impose another condition on $\eps_0.$ 
The eigenvalue of an operator $N_0 + v W + s W_0$ from $\alpha$ 
which is obtained by perturbation of~$z_0$ we denote $z(v,s).$ 
Since $z_0$ is simple, there can be no branching of the function $z(v,s)$ which is therefore single-valued in a small enough neighbourhood of $(0,0).$

The set $z(0,s_0 \partial \mbD)$ is a circle in the~$z$-plane with centre at $z_0 = z(0,0)$ 
and radius~$s_0.$ We choose $\eps_0$ small enough so that all $z(\eps_0\mbD,0)$ are inside 
this circle. 
Now, we fix $v \in \bar \mbD$ and deform~$s$~in 
\[
  z(\eps_0 v, s \partial \mbD)
\]
from~$s_0$ to zero. Clearly, for all $v \in \mbD$ these are closed loops. 
For $s = s_0$ and small enough fixed~$v$ this loop encloses~$z_0$ and for $s = 0$ this loop degenerates to a point 
$z(\eps_0 v,0)$ different (by assumption) from~$z_0$. Hence, for some $s$ between~$s_0$ and $0$ the loop has 
to cross~$z_0.$ The claim is proved. 

The intersection of $\alpha$ with the resonance set must be a union of smooth curves.
Since~$z_0$ has geometric multiplicity~$1$ at $N_0$, such a curve is unique and simple near $N_0$. 

\medskip
The eigenvalue $z = z(v,s)$ depends analytically on~$v$ and~$s$ 
and $\frac{\partial z}{\partial s}(0,0) = 1 \neq 0.$ 
Therefore, the implicit equation $z_0 = z(v,s(v))$ defines a single-valued analytic function $s(v)$ 
with $s(0) = 0.$ 

%The dependence of this~$s$ on~$t$ is analytic and so this gives an analytic curve. 
%
% If the multiplicity of~$z_0$ is~$1,$ then for each $t \in \mbD$ there exists only one resonant $s\in\mbD,$ 
% by the stability of isolated eigenvalues. 
% In general, the curve $s(t)$ is also simple, since the other $m-1$ eigenvalues stay frozen to~$z_0,$ 
% (be careful, --- a quote from ``Hunger games'').
\end{proof}
%We denote this simple curve $\gamma_\chi.$

% Further, as in~\cite[Theorem 4.2.3]{Aza17} one can show that $\gamma_\chi(0)$ and $\chi$ are co-linear:
% \begin{thm} \label{T: chi(0) is exactly that}
%   If $\chi(v)$ is an eigenpath corresponding to $\gamma_\chi(v),$ then $\chi(0)$ 
%   is co-linear to $\chi.$ 
% \end{thm}

For a simple eigenvalue~$z_0=z(0)$ the operator~$W_0$ is uniquely determined,
and therefore, in this case so is the curve~$\gamma$ from Theorem~\ref{T: intersection is a curve},
which thus depends only on the triple $(z_0,N_0,W).$ We denote this simple curve by 
\begin{equation} \label{F: gamma(z,N,W)}
  \gamma(z_0,N_0,W).
\end{equation}
%Perturbing~$z_0$ gives a family of curves $\gamma(z(v),N_v,W).$
%This map has the following properties:
%$\gamma_{z\lra s}(t)$ is holomorphic in $z,\ s$ and $t;$ 
%$z$ is an eigenvalue of $\gamma_{z\lra s}(t)$ for all $t,$
%and for any $s$ the operator $\gamma_{z\lra s}$ is a linear combination of $H_0,$~$W$,
%and a rank-one operator $W_0(s) = \scal{\chi(s),\cdot}\chi(s),$ where $H_s\chi(s)=z(s)\chi(s).$ 
%As we shall soon see, for any fixed pair $z \lra s$ the direction~$W$ is not tangent
%to the curve $s \mapsto \gamma_{z\lra s}.$ 
%So, we can write 
%\[
%  \gamma_{z\lra s}(t) = H_s + (t-s)W + \alpha(s)(t-s)W_0 + O((t-s)^2), \ t \to s,
%\]
%where $\alpha(s) \neq 0.$ 
%What is going on here?

\bigskip
% Given Theorem~\ref{T: chi(0) is exactly that}, 
Proof of the next theorem follows that of~\cite[Theorem~4.3.1]{Aza17} with some adjustments. 
In~this theorem we assume that the geometric multiplicity $m$ is equal to $1,$ and so we can and do write $z(v)$ etc., instead of $z_\tau(v)$~etc. 
However we note that this assumption is not necessary for the proof as long as we have the curve~$\gamma,$ the existence of which is implied by the condition $m=1$ according to Theorem~\ref{T: intersection is a curve}.

\begin{thm} \label{T: high order then tangent}
%Under Assumption~\ref{A: Main Assumption}
Let~$k$ be an integer $\geq 2.$
Let~$z(v)$ be a simple eigenvalue function of $N_v$, $\phi(v)$ a corresponding eigenvector function, and let $\chi_0 = \phi(0).$  
Let $\gamma$ be the curve \eqref{F: gamma(z,N,W)}.
If~$\chi_0$ is an eigenvector of depth $\geq k-1$ for the triple $(z_0; N_0,W),$
then 
\begin{enumerate}
  \item the direction~$W$ is tangent to the curve $\gamma$ to order at least $k,$
  \item in the Taylor expansion 
\begin{equation} \label{F: chi(s)=chi0+s*chi1+...}
  \chi(v) = \sum_{j=0}^\infty v^j \chi_j 
\end{equation}
of an eigenpath $\chi(v)$ corresponding to $\gamma(v),$ 
the vectors 
$
  \chi_0, \chi_1, \ldots, \chi_{k-1}
$
are resonance vectors of orders respectively $1,2,\ldots,k,$
  \item for any eigenpath \eqref{F: chi(s)=chi0+s*chi1+...} corresponding to 
$\gamma$ and for all $j=1,2,\ldots,k-1$ the vector
$\bfA_{z_0}(N_0,W) \chi_j$ is a linear combination of vectors
$\chi_0, \chi_1, \ldots, \chi_{j-1}.$ Moreover, if the parametrisation 
of the curve $\gamma$ is standard then 
$
  \bfA_{z_0}(N_0,W) \chi_j = \chi_{j-1}.
$
  \item the vectors 
$
  \chi_0, \chi_1, \ldots, \chi_{k-2},
$
have depth at least one and are~$W$-orthogonal to $\clV(\bar z_0,N_0^*).$
\end{enumerate} 
\end{thm}
\begin{proof}
Let $N(v)$ be a parametrisation of the resonance curve $\gamma.$ 
Since the curve $\gamma$ is the intersection of the resonance set with the affine plane 
$N_0 + \mbC W + \mbC W_0,$ where $W_0 = \scal{\chi_0,\cdot}\chi_0,$ 
the Taylor expansion of $N(v)$ has the form
\begin{equation} \label{F: H(s)=H0+s*(alpha1*V+beta1*W)+...}
  N(v) = N_0 + \sum_{j=1}^\infty v^j(\alpha_j W + \beta_j W_0).
\end{equation}
% By Theorem~\ref{T: chi(0) is exactly that}, 
A path of eigenvectors~$\chi(v)$ corresponding to this path has Taylor series
\[
  \chi(v) = \chi_0 + v\chi_1 + v^2\chi_2 + \ldots
\]
which starts at the vector~$\chi_0.$ 
Comparing the coefficients of~$v^1$ on both sides of the eigenvalue equation
\begin{equation} \label{F: N(v)chi(v)=z0 chi(v)}
 N(v) \chi(v) = z_0 \chi(v), 
\end{equation}
gives 
\[
  (N_0-z_0)\chi_1 = -(\alpha_1 W + \beta_1 W_0)\chi_0.
\]
The vector $(N_0-z_0)\chi_1$ is orthogonal to the eigenspace~$\clV(\bar z_0,N_0^*)$ 
and in particular it is orthogonal to
\[
  \chi^*_0 := \phi^*(0).
\] 
Hence,
\[
  \scal{\chi^*_0,(\alpha_1 W + \beta_1 W_0)\chi_0} = 0.
\]
Since by the premise~$\chi_0$ has depth $\geq 1,$ 
we have \[\scal{\chi^*_0,W\chi_0} = 0,\] and
combining this equality with previous one gives
$
  \beta_1 = 0.
$
Combining this with \eqref{F: H(s)=H0+s*(alpha1*V+beta1*W)+...} 
we see that $N(v)$ is tangent to~$W$ at~$N_0$ (to order at least 2), and 
\[
  (N_0-z_0)\chi_1 = -\alpha_1 W \chi_0.
\]
So, by Lemma~\ref{L: this is enough}, the vector $\chi_1$ is a resonance vector, has order~$2$ and 
$
  \bfA_{z_0}(N_0,W) \chi_1 = \alpha_1 \chi_0.
$
%(actually, we cannot conclude this directly, but the argument of proof of Lemma 3.7 shows that this is true).
Further, if the parametrisation of $\gamma(v)$ is standard, then $\alpha_1=1.$ 

This proves the theorem in the case of $k=2.$ We proceed by induction on~$k.$
So, assume that the claim holds for~$k\leq l$
and let~$\chi_0$ be an eigenvector of depth~$\geq l.$
By $l$~times differentiating the eigenvalue equation~\eqref{F: N(v)chi(v)=z0 chi(v)}, we obtain
\[
  \sum_{j=0}^l \binom{l}{j} N^{(j)}(v) \chi^{(l-j)}(v) = z_0 \chi^{(l)}(v).
\]
Letting $v = 0$ and replacing~$\chi^{(j)}(0)/j!$ by~$\chi_j$ gives 
\[
  N_0 \chi_l + \sum_{j=1}^l (\alpha_j W + \beta_j W_0) \chi_{l-j} = z_0 \chi_{l}.
\]
By the induction assumption, we have 
\begin{equation} \label{F: beta1=...=beta(l-1)=0}
  \beta_1 = \ldots = \beta_{l-1} = 0.
\end{equation}
Hence,
\begin{equation} \label{F: (N0-z0)chi(l)+...=0}
  (N_0 - z_0)\chi_l + \sum_{j=1}^{l-1} \alpha_j W \chi_{l-j} + (\alpha_l W + \beta_l W_0) \chi_0 = 0.
\end{equation}
Since~$\chi_0$ has depth at least~$l,$ 
by Corollary~\ref{C: depths equal} so does $\chi^*_0$ and therefore,
for some vector~$g$ we have
$
  \chi^*_0 = \bfA^l_{\bar z_0}(N_0^*,W) g.
$
Since, by the induction assumption,~$\chi_j,$ $j=0,1,\ldots,l-1,$ is a vector of order $j+1,$ 
it follows that for all $j = 0,1,\ldots,l-1$
\[
  \scal{\chi^*_0,W\chi_j} = \scal{\bfA^l_{\bar z_0} g,W\chi_{j}} 
  = \scal{g,W\bfA^l_{z_0} \chi_{j}} = 0.
\]
Hence, it follows from \eqref{F: (N0-z0)chi(l)+...=0} (and $\scal{\chi^*_0, \chi_0}\neq 0$, see \eqref{F: (phi* mu,phi nu)=delta})
by taking the scalar product with~$\chi^*_0$ that 
\begin{equation} \label{F: beta(l)=0}
  \beta_l = 0
\end{equation}
and so
\begin{equation} \label{F: fucking equality}
  (N_0 - z_0)\chi_l + \sum_{j=1}^{l} \alpha_j W \chi_{l-j} = 0.
\end{equation}
It follows from \eqref{F: H(s)=H0+s*(alpha1*V+beta1*W)+...}, \eqref{F: beta1=...=beta(l-1)=0} and 
\eqref{F: beta(l)=0} that~$W$ is tangent to $\gamma$ to order at least~$l+1.$
Further, the vector $(N_0-z_0)\chi_l$ is orthogonal to $\clV(\bar z_0,N_0^*)$
and the vectors 
$
  W\chi_0, \ \ldots, W\chi_{l-2} 
$
are also orthogonal to $\clV(\bar z_0,N_0^*)$ by the induction assumption, 
since the vectors~$\chi_0,\ldots,\chi_{l-2}$ have depth at least one. 
Hence, according to \eqref{F: fucking equality}, so is the vector $W \chi_{l-1}$. 
Using an argument from the proof of Lemma~\ref{L: this is enough}, one can infer 
from \eqref{F: fucking equality} that $\chi_l$ is a resonance vector and 
\[
  \bfA_{z_0}(N_0,W) \chi_l = \sum_{j=1}^{l} \alpha_j \chi_{l-j}.
\]
It also follows from this that~$\chi_l$ has order~$l+1.$ 
Since, by the induction assumption, the vectors~$\chi_0, \ldots, \chi_{l-2}$ 
have depth $\geq 1,$ the last equality implies that~$\chi_{l-1}$ is also a resonance vector
and has the same property. 

Further, if the parametrisation of $\gamma(v)$ is standard, then $\alpha_2 = \ldots = \alpha_{l} = 0$
and $\alpha_1 = 1$ and therefore
$
  \bfA_{z_0}(N_0,W) \chi_l = \chi_{l-1}.
$
\end{proof}

%Is this section worth having?
\section{Tangency of the vector field generated by a commutator}\label{S: example}
To conclude we remark on the well-known fact of the isospectrality of a solution to a differential equation with a Lax representation, from the point of view of this paper.
Let a pair of operators~$N(t)$ and~$W(t)$, depending on time~$t$, satisfy Lax's equation 
\begin{equation}\label{F: Lax rep}
 \frac{d}{dt}N(t) = [N(t),W(t)],
\end{equation}
where $[N,W] = NW - WN$ is the commutator. 
As suggested by the notation, we consider the case when~$N$ and~$W$ respectively take values in the affine space~$\clA$ and the vector space of directions~$\clA_0$. 

Given~$W$, the isospectrality of a solution~$N$ of~\eqref{F: Lax rep} can be viewed as the tangency of the vector field $N\mapsto[N,W]$, hence the curve~$N$, to the resonance set $\euR(z_0)$ for any simple eigenvalue~$z_0$ of $N(t_0)$. 
%Indeed, isospectrality means that any eigenvalue remains fixed and hence by Theorem ? and ? the curve is tangent to the resonance set.
%Tangency to the resonance set is easy to prove using the characterisations (?) of Theorem ?...
Indeed, if for any $t=t_0$ with $N_0 := N(t_0)$ we have $N_0\phi = z_0\phi$, then
\begin{equation*}
%\begin{split}
 \scal{\phi^*,[N_0,W(t_0)]\phi} = \scal{N_0^*\phi^*,W(t_0)\phi} - \scal{\phi^*,W(t_0)N_0\phi} 
% \\ &= z_0(\scal{\phi^*,W(t_0)\phi} - \scal{\phi^*,W(t_0)\phi})  
= 0.
%\end{split}
\end{equation*}
Thus by Theorem~\ref{T: TFAE to chi has order k}, $\phi$~has depth at least one and equivalently $z'(t_0) = 0$. 
Therefore $N(t)$ is tangent at~$t_0$ to the resonance set~$\euR(z_0)$ by Theorem~\ref{T: high order then tangent}. 
Since $t_0$ was chosen generically, it follows that the eigenvalue $z(t)$ remains fixed at~$z_0$ and that $N(t)$ remains tangent to the resonance set~$\euR(z_0)$.

{\itshape Acknowledgements.} 
We thank Yang Shi for bringing our attention to the example involving Lax pairs. 
Also, the first author N.\,A.~thanks his wife, Feruza, for financial support during the work on this paper.

\end{document}